\documentclass{article}
\usepackage[utf8]{inputenc}
\usepackage{geometry}
\pdfoutput=1
\usepackage{amsmath}
\usepackage{amssymb}
\usepackage{amsfonts}
\usepackage{caption}
\usepackage{subcaption}
\usepackage{float}
\usepackage{parskip}
\usepackage{amsthm}
\usepackage{thmtools}
\usepackage{csquotes}
\usepackage{mathdots}
\usepackage[dvipsnames]{xcolor}
\usepackage[shortlabels]{enumitem}
\usepackage{authblk}
\usepackage{graphicx}
\usepackage{glossaries}
\graphicspath{ {./documents/} }
\usepackage{tikz}
\usetikzlibrary{calc}
\usetikzlibrary{decorations.markings}
\usepackage{hyperref}
\hypersetup{
    colorlinks=true,
    linkcolor=BrickRed,
    filecolor=magenta,      
    urlcolor=cyan,
    pdftitle={Overleaf Example},
    pdfpagemode=FullScreen,
    citecolor=JungleGreen,
    }

\newtheorem{theorem}{Theorem}[section]
\newtheorem{proposition}[theorem]{Proposition}

\newtheorem{lemma}[theorem]{Lemma}

\newtheorem{innercustomthm}{Theorem}
\newenvironment{customthm}[1]
  {\renewcommand\theinnercustomthm{#1}\innercustomthm}
  {\endinnercustomthm}

\theoremstyle{definition}

\newcommand{\seqj}[3]{(#1_j)_{j=#2}^{#3}}

\newcommand{\orbs}[2]{\Omega_{#1}(#2)}

\newcommand{\Jfree}{J_{\mathrm{free}}}
\newcommand{\Jfixedone}{J_{\mathrm{fixed},1}}
\newcommand{\Jfixedzero}{J_{\mathrm{fixed},0}}

\newcommand{\bqe}[1]{f^*_{#1}}

\newcommand{\hdim}{\dim_{\mathrm{H}}}

\title{On finitely many base $q$ expansions}
\author[1]{Simon Baker}
\author[2]{George Bender}
\affil[1]{Department of Mathematical Sciences,
Loughborough University,
Loughborough,
LE11 3TU,
UK, 
simonbaker412@gmail.com}
\affil[2]{School of Mathematics,
University of Birmingham,
Edgbaston,
Birmingham,
B15 2TT,
UK,
georgew.bender@gmail.com}

\begin{document}
\setlength{\parskip}{2.3ex}

\maketitle

\begin{abstract}
Given some integer $m \geq 3$, we find the first explicit collection of countably many intervals in $(1,2)$ such that for any $q$ in one of these intervals, the set of points with exactly $m$ base $q$ expansions is nonempty and moreover has positive Hausdorff dimension.
Our method relies on an application of a theorem proved by Falconer and Yavicoli \cite[Theorem 6]{FalYav2022}, which guarantees that the intersection of a family of compact subsets of $\mathbb{R}^d$ has positive Hausdorff dimension under certain conditions.
\end{abstract}

\section{Introduction}

Let $q \in (1,2)$ and let $I_q = [0 , \frac{1}{q-1}]$.
We say that a sequence $\seqj{\epsilon}{1}{\infty} \in \{0,1\}^\mathbb{N}$ is a base $q$ expansion of $x \in I_q$ if
$$x = \sum_{j=1}^\infty \epsilon_j q^{-j}.$$
It is a straightforward exercise to show that $x$ has a base $q$ expansions if and only if $x \in I_q$.
In order to put our work in context, we observe where the research in base $q$ expansions started, how it developed, and what current open questions are receiving attention.
Initially, base $q$ expansions were studied via the dynamical system 
$T(x) = qx \pmod{1}$ on $[0,1]$.
Each point $x \in [0,1]$ has a unique orbit in this system which corresponds to what is known as the \textit{greedy expansion} of $x$.
In an early paper by R\'enyi \cite{renyi1957representations} it is shown that there exists a $T$-invariant probability measure, equivalent to Lebesgue, which is also ergodic.
Following this, Parry \cite{parry1960beta} gave sufficient conditions on a sequence $\seqj{\epsilon}{1}{\infty}$ in order for it to appear as a greedy base $q$ expansion.
We note that these results were actually proved in the more general setting where $q$ is some real number greater than $1$, where the sequences $\seqj{\epsilon}{1}{\infty}$ are elements of $\{0 , 1 , \ldots , \lfloor q \rfloor \}^\mathbb{N}$ instead of $\{0,1\}^\mathbb{N}$.
In the early 90s, more attention was given to the broader system in which all possible base $q$ expansions of $x \in I_q$ are considered.
This is the system in which we are interested in this paper.

We define the set of base $q$ expansions of $x \in I_q$ by
$$\Sigma_q(x) = \left\{ \seqj{\epsilon}{1}{\infty} \in \{0,1\}^\mathbb{N} : x = \sum_{j=1}^\infty \epsilon_j q^{-j} \right\},$$
and the projection map $\pi_q: \{0,1\}^\mathbb{N} \rightarrow I_q$ by
$$\pi_q(\seqj{\epsilon}{1}{\infty}) = \sum_{j=1}^\infty \epsilon_j q^{-j}.$$

The foundational results in base $q$ expansions demonstrate that they behave very differently to the well-understood integer base expansions.
Recall the familiar example of binary expansions, where every point $x \in [0,1]$ has a unique expansion apart from a countable set of points which have exactly two expansions, by virtue of the fact that $\pi_2(10^\infty) = \pi_2(01^\infty)$.
The same is true for all integer base expansions.
By comparison, Sidorov proved that for every $q \in (1,2)$, Lebesgue almost every $x \in I_q$ will have uncountably many base $q$ expansions \cite{Sidorov2003}.
Moreover, given any $m \in \mathbb{N} \cup \aleph_0$, we can choose $q \in (1,2)$ and $x \in I_q$ such that $x$ has exactly $m$ base $q$ expansions \cite{Bak2015OnSmall, sidorov2009expansions, BakSid2014}.
Define 
$$\mathcal{U}_q = \{ x \in I_q : |\Sigma_q(x)| = 1 \},$$
and for $m \geq 2$
$$\mathcal{U}_q^{(m)} = \{x \in I_q : |\Sigma_q(x)| = m \}.$$
In this paper, our results concern the sets $\mathcal{B}_m$, where
$$\mathcal{B}_m = \{ q \in (1,2) : \mathcal{U}_q^{(m)} \neq \emptyset \}.$$
For all $k \geq 2$, we define the $k$-Bonacci number $q_k$ to be the unique number in $(1,2)$ which satisfies
$$q_k^k - q_k^{k-1} - \cdots - q_k - 1 = 0.$$
For $k = 2$, we denote $q_2 = \frac{1 + \sqrt{5}}{2}$ by $G$, the Golden Ratio.
This family of numbers will play an important role in our results and in the discussion below.

Particular attention was given to the number of expansions of $x = 1$ for different values of $q$.
In particular, it was shown by Erd\"os et al. \cite[Theorem 1]{ErJoHo1991} that the set of $q$ for which $1$ has a unique base $q$ expansion has uncountably many elements, while Dar\'oczy and K\'atai \cite[Theorem 3]{DarKat1995Structure}
showed that it has zero Lebesgue measure, and Hausdorff dimension $1$.
Moreover, in \cite{ErdJoo1993} Erd\"os and Joo showed that there are uncountably many bases $q$ for which $1 \in \mathcal{U}_q^{\aleph_0}$.
Komornik and Loreti \cite{KomLor98} found an algebraic construction for the value $q_{\mathrm{KL}} \approx 1.78723$, now known as the Komornik-Loreti constant, which is the smallest base for which $1$ has a unique expansion.
Later, it was found that $q_{\mathrm{KL}}$ plays an important role in our understanding of $\mathcal{U}_q$ \cite{GlenSid2001}.

Generalising the investigation of the number of expansions of $x=1$ for different bases $q$, we can ask for which values of $q$ there exists at least one $x \in I_q$ that has a unique expansion in base $q$.
It was shown by Erd\"os et al \cite{erdos1990characterization} that if $q \in (1,G)$ then all\footnote{It is obvious that the endpoints of $I_q$, $0 = \pi_q(0^\infty)$ and $\frac{1}{q-1} = \pi_q(1^\infty)$, have unique expansions.} $x \in (0,\frac{1}{q-1})$ have uncountably many base $q$ expansions.
Sidorov and Vershik \cite{SidVer1998Goldenshift} showed that if $q = G$ then all $x \in (0, \frac{1}{q-1})$ have uncountably many expansions unless $x = nG \pmod{1}$ for some $n \in \mathbb{Z}$, in which case $x \in \mathcal{U}_G^{\aleph_0}$.
In \cite{GlenSid2001} Glendinning and Sidorov proved the following dichotomy.
If $q \in (G , q_{\mathrm{KL}})$ then $\mathcal{U}_q$ is countably infinite and every unique expansion is eventually periodic, but if $q \in (q_{\mathrm{KL}}, 2)$ then $\mathcal{U}_q$ contains a Cantor set with positive Hausdorff dimension.
For $q= q_{\mathrm{KL}}$, it is shown that $\mathcal{U}_q$ is uncountably infinite and has zero Hausdorff dimension.
We will see in more detail later (Lemma \ref{lemma: containment in U_q}) that if $q \in (G,2)$ then 
$\dim_{\mathrm{H}}(\mathcal{U}_q) \nearrow 1 $ as $q \nearrow 2$ \cite{GlenSid2001}.
Another avenue of research concerns the function $q \mapsto \hdim(\mathcal{U}_q)$ which is shown by Komornik et al in \cite{KomKonLi2017Hausdorff} to be continuous, have bounded variation, and resemble the Cantor function. 
Further work was carried out on the properties of this function by the first author et al in \cite{BarBakKon2019Entropy}.
Together these results provide us with a reasonably complete understanding of $\mathcal{U}_q$ for all $q \in (1,2)$.
Following on from the study of $\mathcal{U}_q$ it is natural to ask, given $m \in \mathbb{N}$, for which bases $q$ is the set $\mathcal{U}_q^{(m)}$ nonempty and what is its Hausdorff dimension?
In \cite{sidorov2009expansions} Sidorov shows that the smallest element of $\mathcal{B}_2$ is $q_m \approx 1.71064$, which is the root in $(1,2)$ of
$$x^4 = 2x^2 + x + 1.$$
The next smallest element of $\mathcal{B}_2$ is $q_f \approx 1.75488$ which is the root in $(1,2)$ of
$$x^3 = 2x^2 - x + 1,$$
and it is shown by the first author in \cite{BakSid2014} that $q_f$ is also the smallest element of $\mathcal{B}_m$ for all $m \geq 3$.
Sidorov \cite{sidorov2009expansions} proved the following key results concerning $\mathcal{B}_2$ and it is these results which we build on directly in this paper.
After proving
\begin{equation}
    \label{equation: q in B_2 condition}
    q \in \mathcal{B}_2 \iff 1 \in \mathcal{U}_q - \mathcal{U}_q,
\end{equation}
it is a straightforward application of Newhouse's Theorem \cite{Newhouse1970} (see Subsection \ref{subsection: thickness and interleaving}) to prove that $[q_3, 2) \subset \mathcal{B}_2$, where $q_3 \approx 1.83929$.
Sidorov generalises this theorem, claiming that there exists some $\gamma_m > 0$ such that $[2- \gamma_m , 2) \subset \mathcal{B}_m$ for all $m \geq 3$.
Unfortunately, the authors believe this proof contains a mistake which cannot be fixed.
If this is true, and this result is invalid, then the question of whether the Lebesgue measure of $\mathcal{B}_m$ is positive is still open for $m \geq 3$.
The main theorem of this paper (Theorem \ref{theorem: main theorem}) is a stronger version of Sidorov's general result in the sense that we are able to find the first explicit intervals contained in $\mathcal{B}_m$ for $m \geq 3$.
More recently it was proved by the first author and Zou \cite{BakerYuru2023Metric} that for Lebesgue almost every $q \in (q_\mathrm{KL}, M+1]$ we have
$\hdim \mathcal{U}_q^{(m)} \leq \max \{ 0 , 2 \hdim \mathcal{U}_q -1 \}$
for all $m \in \{2,3, \ldots \}$.
This result improves upon the bound of $\hdim \mathcal{U}_q^{(m)} \leq \hdim \mathcal{U}_q$ for all $m \in \{2,3, \ldots \}$ which can be seen by a simple branching argument (see \cite{sidorov2009expansions} or \cite{Bak2015OnSmall}).
The equivalence \eqref{equation: q in B_2 condition} illustrates how information about $\mathcal{U}_q$ can be used to study $\mathcal{B}_2$.
More generally, we will see in Subsection \ref{subsection: a sufficient condition for q in B_m} that we can guarantee that $q \in \mathcal{B}_m$ if a certain intersection of affine images of the set $\mathcal{U}_q$ is nonempty.
Moreover, via an application of a theorem of Falconer and Yavicoli \cite[Theorem 6]{FalYav2022}, we are able to bound the Hausdorff dimension of the corresponding $\mathcal{U}_q^{(m)}$ set from below.
The version of this theorem that we use appears as Theorem \ref{theorem: FandY} in Subsection \ref{subsection: thickness and interleaving}.


Our most general result is Theorem \ref{theorem: main theorem} which locates an explicit collection of countably many intervals in $\mathcal{B}_m$ for all $m \geq 3$.
More precisely, for any $m \geq 3$, Theorem \ref{theorem: main theorem} tells us that there is a neighbourhood of $q_k$ contained in $\mathcal{B}_{m}$ whenever $k$ is sufficiently large.
In the case of $\mathcal{B}_3$ we are able to obtain a stronger result which is the content of Theorem \ref{theorem: B_3 theorem}.

\begin{customthm}{A}
    \label{theorem: main theorem}
    Let $m \in \mathbb{N}$ and let $k \geq K_m$ where 
    $$K_m = \left\lceil \frac{20}{19}\left( \log_{1.999}(m+2) + 24 \right) + 4 \right\rceil.$$
    If $|q - q_k| < q_k^{-(m+2)k-3}$, then $q \in \mathcal{B}_{m+2}$ and 
    $\hdim(\mathcal{U}_q^{(m+2)}) \geq 1 - 1024(m+2)^{\frac{20}{19}}q^{4-k} > 0$.
\end{customthm}

\begin{table}[h]
\label{table: intervals for main theorem}
\begin{center}
\renewcommand{\arraystretch}{1.2}
\begin{tabular}{|c|c|l|l|l|}
    \hline
    $m$ & $K_m$ & $q_{K_m}$ (approx.) & Interval in $q$-space contained in $\mathcal{B}_{m+2}$ & $\displaystyle \hdim(\mathcal{U}_{q_{K_m}}^{(m+2)})$ \\
    \hline
    1 & 31 & $1.999999999534342$ & $[q_{31} - 1.26218 \times 10^{-29} , q_{31} + 1.26218 \times 10^{-29}]$ &  $\geq 
    0.999967173 \ldots$ \\
    2 & 32 & $1.999999999767168$ & $[q_{32} - 3.67342 \times 10^{-40} , q_{32} + 3.67342 \times 10^{-40}] $ & $\geq 
    0.999983586 \ldots$ \\
    3 & 32 & $1.999999999767168$ & $[q_{32} - 8.55285 \times 10^{-50}  , q_{32} + 8.55285 \times 10^{-50}]$ & $\geq 
    0.999979240 \ldots $ \\
    4 & 32 & $1.999999999767168$ & $[q_{32} - 1.99136 \times 10^{-59} , q_{32} + 1.99136 \times 10^{-59}]$ & $\geq 
    0.999974848 \ldots$ \\
    5 & 33 & $1.999999999883594$ & $[q_{33} - 3.62227 \times 10^{-71}, q_{33} + 3.62227 \times 10^{-71}]$ & $\geq 
    0.999985209 \ldots$ \\
    \hline
\end{tabular}
\caption{For small $m$, the table shows the leftmost interval contained in $\mathcal{B}_{m+2}$ and the lower bound on $\hdim(\mathcal{U}_{q_{K_m}}^{(m+2)})$, provided by Theorem \ref{theorem: main theorem}.
We emphasise that these are the leftmost intervals guaranteed to exist by Theorem \ref{theorem: main theorem} there may be other intervals in $\mathcal{B}_{m+2}$ contained to the left of these.}
\end{center}
\end{table}


\begin{customthm}{B}
    \label{theorem: B_3 theorem}
    \begin{enumerate}[(a)]
        \item
        \label{theorem item: B_3 theorem at least 10}
        For $k \geq 10$, if $|q - q_k| \leq q_k^{-2k-6}$ then $q \in \mathcal{B}_3$.
        \item 
        \label{theorem item: B_3 theorem 9}
        If $0 < q - q_9 \leq q_9^{-24}$ then $q \in \mathcal{B}_3$.
    \end{enumerate}
\end{customthm}

\begin{table}[h]
\label{table: intervals for B_3 theorem}
\begin{center}
\renewcommand{\arraystretch}{1.2}
\begin{tabular}{|c|l|l|l|}
    \hline
    $k$ & $q_k$ (approx.) & Interval in $q$-space contained in $\mathcal{B}_3$  \\
    \hline
    9 & $1.99802947026229$ & $[q_9, q_9 + 6.10316 \times 10^{-8} ]$ \\
    10 & $1.99901863271010$ & $[q_{10} - 1.50925 \times 10^{-8} , q_{10} + 1.50925 \times 10^{-8} ]$ \\
    11 & $1.99951040197829$ & $[q_{11} - 3.75092 \times 10^{-9} , q_{11} + 3.75092 \times 10^{-9}]$ \\
    12 & $1.99975550093732$ & $[q_{12} - 9.34745 \times 10^{-10} , q_{12} + 9.34745 \times 10^{-10}] $ \\
    13 & $1.99987783271155$ & $[q_{13} - 2.33286 \times 10^{-10} , q_{13} + 2.33286 \times 10^{-10}]$ \\
    \hline
\end{tabular}
\end{center}
\caption{Neighbourhoods of $q_k$ contained in $\mathcal{B}_3$ for small values of $k$, provided by Theorem \ref{theorem: B_3 theorem}.}
\end{table}

Together Theorems \ref{theorem: main theorem} and \ref{theorem: B_3 theorem} provide the first explicit intervals contained in $\mathcal{B}_m$ for $m \geq 3$.
Since $K_1 = 31$, Theorem \ref{theorem: B_3 theorem} proves the existence of intervals in $\mathcal{B}_3$ where Theorem \ref{theorem: main theorem} does not apply, namely, when $m = 1$ and $9 \leq k < K_1$.
The reason for this discrepancy in the values of $k$ where the theorems apply is due to the number of sets in the intersection which we require to be nonempty.
To be precise, the argument for the proof of Theorem \ref{theorem: B_3 theorem} boils down to proving that the intersection of two compact subsets of $\mathbb{R}$ is nonempty, this allows us to apply Theorem \ref{theorem: Newhouse} \cite{Newhouse1970}.
However, to prove Theorem \ref{theorem: main theorem} for some $m \in \mathbb{N}$, we require a collection of $(m+2)$ sets to be nonempty, hence we cannot directly apply Theorem \ref{theorem: Newhouse} \cite{Newhouse1970}.
In Section \ref{section: preliminaries} we outline the main theorems and propositions used for Theorems \ref{theorem: main theorem} and \ref{theorem: B_3 theorem}, whose proofs are contained in Sections \ref{section: proof of main theorem} and \ref{section: proof of B_3 theorem} respectively.

\section{Preliminaries}
\label{section: preliminaries}

In this section we introduce the results required for the proofs of Theorems \ref{theorem: main theorem} and \ref{theorem: B_3 theorem}.
We draw special attention to Theorem \ref{theorem: FandY} - a special case of a theorem of Falconer and Yavicoli \cite[Theorem 6]{FalYav2022} - on which Theorem \ref{theorem: main theorem} relies.

\subsection{Thickness and interleaving}
\label{subsection: thickness and interleaving}

Let $|X|$ and $\mathrm{conv}(X)$ denote the diameter and convex hull respectively of a set $X \subset \mathbb{R}$.
Given some compact set $C \subset \mathbb{R}$, Newhouse \cite{Newhouse1970} defines its \textit{thickness} $\tau(C)$ via its \textit{gaps} and \textit{bridges}.
A \textit{gap} $G$ of $C$ is a maximal connected component\footnote{Including both the unbounded components to the left and right of $\mathrm{conv}(C)$.} of $\mathbb{R} \setminus C$, and since $C$ is compact, all such gaps are open intervals.
Any compact set $C \subset \mathbb{R}$ admits a stepwise construction removing gaps in order of decreasing diameter.
Precisely, we start by removing the unbounded gaps from $\mathbb{R}$ to leave $\mathrm{conv}(C)$.
From here we remove from $\mathrm{conv}(C)$ the next largest gap $G_1$ (if there is a tie, take any).
This leaves two closed intervals, called \textit{bridges}, $B_1^L$ and $B_1^R$ with the gap $G_1$ in between.
We continue in this way, removing the next largest gap from the remaining union of bridges.
In general, at step $n$, gap $G_n$ with $|G_{n-1}| \geq |G_n| \geq |G_{n+1}|$, is removed from a closed interval, leaving bridges $B_n^L$ and $B_n^R$ either side of $G_n$.
We define the thickness of $C$ by
$$\tau(C) = \inf_n \left\{ \min \left( \frac{|B_n^L|}{|G_n|} , \frac{|B_n^R|}{|G_n|} \right) \right\}. $$
Note that the initial step removing the unbounded gaps of $C$ is not involved in the thickness calculation.
Moreover, it can be checked that the thickness is independent of the choice of gap when two or more gaps have the same diameter
\cite{FalYav2022}.
If $B$ is a bridge of $C$ then all gaps contained within $B$ have diameter no larger than the gaps on either side of $B$.
Note also that any bridge $B$ of $C$ has the property that all gaps of $C$ are either contained in $B$ or contained in the complement of $B$.

We will be interested in the thickness of certain compact subsets of $\mathcal{U}_q$ for $q \in (1,2)$.
To study these effectively we will need to define some notation.
Let $\{0,1\}^*$ be the set of all finite words in $\{0,1\}$ and let $|\delta|$ denote the length of a finite word $\delta \in \{0,1\}^*$.
A sequence $\seqj{\epsilon}{1}{\infty} \in \{0,1\}^\mathbb{N}$ is said to \textit{avoid} a word (finite or infinite) $\seqj{\delta}{1}{m} \in \{0,1\}^m$ if there is no index $i \in \mathbb{N}$ such that $\seqj{\epsilon}{i}{i+m-1} = \seqj{\delta}{1}{m}$.
For any $k \in \mathbb{N}$, define the set $S_k \subset \{0,1\}^\mathbb{N}$ by
$$S_k = \{ \seqj{\epsilon}{1}{\infty} \in \{0,1\}^\mathbb{N} : \seqj{\epsilon}{1}{\infty} \ \mathrm{avoids} \ (01^k) \ \mathrm{and} \ (10^k) \}.$$
The following key lemma is a consequence of the work in \cite[Lemma 4]{GlenSid2001}.
\begin{lemma}
    \label{lemma: containment in U_q}
    Let $k \geq 2$.
    \begin{enumerate}[(a)]
        \item 
        \label{lemma item: U_q contains S_k}
        If $q \in (q_k , 2)$ then $\pi_q(S_k) \subset \mathcal{U}_q$.
         \item 
        \label{lemma item: containment in U_q Cantor set}
        If $q \in (q_k , 2)$ then $\pi_q(S_k)$ is a Cantor set and moreover every gap of $\pi_q(S_k)$ is an open interval of the form
        $$\left(\pi_q(\delta (01^{k-1})^\infty) , \pi_q(\delta (10^{k-1})^\infty)\right),$$
        where $\delta \in \{0,1\}^*$ avoids $(01^k)$ and $(10^k)$.
    \end{enumerate}
\end{lemma}

It is also shown in \cite{GlenSid2001} that $\hdim(\pi_q(S_k)) \rightarrow 1$ as $q \rightarrow 2$, but more important for us is the following result which is a consequence of the estimates in the proof of \cite[Theorem 4.4]{sidorov2009expansions}.
\begin{lemma}
    \label{lemma: thickness of S_k}
    $\tau(\pi_q(S_k)) > q^{k-3}$ whenever $k \geq 4$ and $q > q_k$.
\end{lemma}

We define interleaving along with a stronger version we call $\epsilon$-\textit{strong interleaving}.
Two compact sets $C_1, C_2 \subset \mathbb{R}$ are \textit{interleaved} if neither set is contained in a gap of the other.
The Hausdorff metric, defined on compact subsets of $\mathbb{R}$ is given by
$$d_{\mathrm{H}}(X,Y) = \max\left\{ \sup_{x \in X} d(x , Y) , \sup_{y \in Y} d(y , X) \right\},$$
where $d$ is the normal Euclidean distance on $\mathbb{R}$.
Let $\epsilon > 0$ and let $A$ and $B$ be compact subsets of $\mathbb{R}$.
We say that the sets $A$ and $B$ are $\epsilon$-\textit{strongly interleaved} if the sets $A'$ and $B'$ are interleaved whenever $A'$ and $B'$ are compact subsets of $\mathbb{R}$ with the property that $d_{\mathrm{H}}(A , A') \leq \epsilon$ and $d_{\mathrm{H}}(B , B') \leq \epsilon$.
Observe that if $A$ and $B$ are $\epsilon$-strongly interleaved for some $\epsilon > 0$ then $A$ and $B$ are $\epsilon'$-strongly interleaved for every $0 < \epsilon' < \epsilon$.
It is clear that for any $\epsilon > 0$, if $A$ and $B$ are $\epsilon$-strongly interleaved then they are interleaved.
The following lemma allows us to declare that two compact sets in $\mathbb{R}$ are $\epsilon$-strongly interleaved if they contain points which are sufficiently well-separated.
This will be useful for the proof of Theorem \ref{theorem: B_3 theorem}.

\begin{lemma}
    \label{lemma: strong interleaving}
    Suppose $A$ and $B$ are compact subsets of $\mathbb{R}$ and there exists $a_1 , a_2 \in A$ and $b_1 , b_2 \in B$ with the property that
    $$a_1 \leq b_1 \leq a_2 \leq b_2,$$
    then $A$ and $B$ are interleaved.
    Moreover, if
    $$\min \{ (b_1 - a_1) , (a_2 - b_1) , (b_2 - a_2) \} \geq 2\epsilon,$$
    for some $\epsilon > 0$ then $A$ and $B$ are $\epsilon$-strongly interleaved.
\end{lemma}

\begin{proof}
    Let $A$ and $B$ be compact subsets of $\mathbb{R}$ with $a_1, a_2 \in A$ and $b_1, b_2 \in B$ satisfying $a_1 \leq b_1 \leq a_2 \leq b_2$.
    By definition, $A$ and $B$ are interleaved if neither is contained in a gap of the other.
    Since $a_2 \in A$ we can partition the gaps of $A$ into those that lie to the left of $a_2$ and those that lie to the right of $a_2$.
    $B$ cannot be contained in a gap to the left of $a_2$ because $b_2 \in B$ and $b_2 \geq a_2$.
    Similarly, $B$ cannot be contained in a gap to the right of $a_2$ because $b_1 \in B$ and $b_1 \leq a_2$.
    $A$ cannot be contained in a gap of $B$ by a similar argument partitioning the gaps of $B$ into those to the left and right of $b_1$.

    Suppose $\epsilon > 0$ and $\min \{ (b_1 - a_1) , (a_2 - b_1) , (b_2 - a_2) \} \geq 2\epsilon$.
    Let $A'$ and $B'$ be compact subsets of $\mathbb{R}$ satisfying $d_{\mathrm{H}}(A , A') , d_{\mathrm{H}}(B, B') \leq \epsilon$.
    Therefore, there exists $a_1' , a_2' \in A'$ and $b_1' , b_2' \in B'$ satisfying
    $$|a_1' - a_1| , |a_2' - a_2| , |b_1' - b_1| , |b_2' - b_2| \leq \epsilon.$$
    Hence
    $$(b_1' - a_1') \geq (b_1 - a_1) - |b_1' - b_1| - |a_1' - a_1| \geq 2\epsilon - \epsilon - \epsilon = 0,$$
    $$(a_2' - b_1') \geq (a_2 - b_1) - |a_2' - a_2| - |b_1' - b_1| \geq 2\epsilon - \epsilon - \epsilon = 0,$$
    $$(b_2' - a_2') \geq (b_2 - a_2) - |b_2' - b_2| - |a_2' - a_2| \geq 2\epsilon - \epsilon - \epsilon = 0.$$
    So $a_1' \leq b_1' \leq a_2' \leq b_2'$ and we conclude that $A'$ and $B'$ are interleaved by the first part of the lemma.
\end{proof}

Newhouse \cite{Newhouse1970} proved that information on the thickness and interleaving of two compact sets in $\mathbb{R}$ can be sufficient to imply a nonempty intersection of the sets.
This is the content of Theorem \ref{theorem: Newhouse} and will be used to prove Theorem \ref{theorem: B_3 theorem}.

\begin{customthm}{N}
    \label{theorem: Newhouse}
    Let $C_1$ and $C_2$ be interleaved compact subsets of $\mathbb{R}$ with thicknesses $\tau_1$ and $\tau_2$ respectively.
    If $\tau_1 \tau_2 \geq 1$ then $C_1 \cap C_2 \neq \emptyset$.
\end{customthm}

In our proof of Theorem \ref{theorem: main theorem}, we apply Theorem \ref{theorem: FandY}, which is a special case of a more general theorem proved by Falconer and Yavicoli \cite[Theorem 6]{FalYav2022}.

\begin{customthm}{F\&Y}
\label{theorem: FandY}
For any $m \in \mathbb{N}$, let $\seqj{C}{1}{m+2}$ be compact subsets of $\mathbb{R}$ where each $C_j$ has thickness at least $\tau$.
Assume that
\begin{enumerate}[(i)]
    \item $B = \bigcap_j \mathrm{conv}(C_j)$ is nonempty 
\end{enumerate}    
    and
\begin{enumerate}
    \item[(ii)] there exists $c \in (0,1)$ such that 
    \begin{equation}
        \label{equation: FandY inequality}
        (m+2)\tau^{-c} \leq \frac{1}{(432)^2}\beta^c(1-\beta^{1-c}),
    \end{equation}
    where
    $\beta = \min \left\{ \frac{1}{4} , \frac{|B|}{\max_j |C_j|} \right\}$.
\end{enumerate}
Then $\hdim (\bigcap_j C_j) \geq 1 - 1024(m+2)^{1/c}\tau^{-1} > 0$.
\end{customthm}


\subsection{A sufficient condition for $q \in \mathcal{B}_m$}
\label{subsection: a sufficient condition for q in B_m}

We define the maps $f_0$ and $f_1$ by
\begin{align*}
f_0&:\left[0,\frac{1}{q(q-1)}\right] \rightarrow I_q; &f_0(x) &= qx, \\
f_1&:\left[\frac{1}{q}, \frac{1}{q-1}\right] \rightarrow I_q; &f_1(x) &= qx - 1,
\end{align*}
and set $E_q = \{f_0 , f_1\}$ (Figure \ref{figure: base q expansion}).
For any map $h$, we denote its inverse by $h^{-1}$ and the $i$-fold composition $\underbrace{h \circ \cdots \circ h}_{i \ \mathrm{times}}$ is denoted by $h^i$. 

\begin{figure}[h]
    \centering
        \begin{subfigure}[t]{0.45\textwidth}
        \vskip 0pt
     \begin{tikzpicture} [scale = 7]
    \begin{scope}[thick]
        \draw (0,0) -- (1,0);
        \draw (0,0) -- (0,1);
        \draw (1,0) -- (1,1);
        \draw (0,1) -- (1,1);
        
        \draw [dotted] (0,0) -- (1,1);

        \draw [blue] (0,0) -- (0.4,0.67);
        \draw (0.4,0.67) -- (0.6,1);
        \draw [blue] (0.4,0) -- (0.67,0.444);
        \draw (0.67,0.444) -- (1,1);
        
        \draw [dotted] (0.6,1) -- (0.6,0);
        \draw [dotted] (0.4,0) -- (0.4,1);

        \draw [dashed] (0,0.67) -- (0.67,0.67);
        \draw [dashed] (0.67,0.67) -- (0.67,0);
    \node[above] at (0.2,1){\small $f_0$};
    \node[above] at (0.8,1){\small $f_1$};

    \node[below] at (0,0){\small $0$};
    \node[below] at (0.4,0){\small $\frac{1}{q}$};
    \node[below] at (0.6,0){\small $\frac{1}{q(q-1)}$};
    \node[below] at (1,0){\small $\frac{1}{q-1}$};

    \filldraw[black] (0.6,1) circle (0.2pt);
    \draw (0.6,1) circle (0.2pt);
    \filldraw[black] (0.4,0) circle (0.2pt);
    \filldraw[black] (0,0) circle (0.2pt);
    \filldraw[black] (1,1) circle (0.2pt);
    \end{scope}
    
        \end{tikzpicture}
            \caption{The maps generating base $q$ expansions.}
    \label{figure: base q expansion}
        \end{subfigure}
        \hfill
        \begin{subfigure}[t]{0.45\textwidth}
        \vskip 0pt
 \begin{tikzpicture} [scale = 7]
    \begin{scope}[thick]
        \draw (0,0) -- (1,0);
        \draw (0,0) -- (0,1);
        \draw (1,0) -- (1,1);
        \draw (0,1) -- (1,1);
        
        \draw [dotted] (0,0) -- (1,1);
        
        
        \draw [dotted] (0.6,1) -- (0.6,0);


        \draw [blue] (0,0) -- (0.6,1);
        \draw [blue] (0.6,0) -- (1,0.67);

    \node[above] at (0.2,1){\small $f_0$};
    \node[above] at (0.8,1){\small $f_1$};

    \node[below] at (0,0){\small $0$};
    \node[below] at (0.6,0){\small $\frac{1}{q}$};
    \node[below] at (1,0){\small $1$};

    \filldraw[white] (0.6,1) circle (0.2pt);
    \draw (0.6,1) circle (0.2pt);
    \filldraw[black] (0.6,0) circle (0.2pt);
    \filldraw[black] (0,0) circle (0.2pt);
    \filldraw[black] (1,0.67) circle (0.2pt);
    \end{scope}
    
        \end{tikzpicture}            
        \caption{The special case of the greedy expansion given by $f(x) = qx \pmod{1}$ on $[0,1]$.}
        \label{figure: greedy base q expansions}
        \end{subfigure}
        \caption{}
        \label{figure: variants of base q expansions}
\end{figure}

Note that $f_0$ and $f_1$ depend on $q$ implicitly, but in all cases the value of $q$ will be clear from context, so we suppress this in the notation.
The following lemma is a routine check from the definitions of the maps $f_0$, $f_1$ and $\pi_q$, and builds the picture of how the maps of $E_q$ `generate' the base $q$ expansions.

\begin{lemma}
    \label{lemma: relation between maps f and base q expansions}
    For any $q \in (1,2)$, $\seqj{\delta}{1}{m} \in \{0,1\}^*$, $\seqj{\epsilon}{1}{\infty} \in \{0,1\}^\mathbb{N}$ we have,
    $$f_{\delta_1}^{-1} \circ \cdots \circ f_{\delta_m}^{-1} (\pi_q(\seqj{\epsilon}{1}{\infty})) = \pi_q( \seqj{\delta}{1}{m} \seqj{\epsilon}{1}{\infty}).$$
\end{lemma}
The map $\pi_q$ also applies to the extended alphabet $\{-1 , 0 , 1\}^\mathbb{N}$.
Define $I_q^* = [ \frac{-1}{q-1} , \frac{1}{q-1} ]$.
If $\seqj{\epsilon}{1}{\infty} \in \{-1 , 0 , 1\}^\mathbb{N}$ then $\pi_q(\seqj{\epsilon}{1}{\infty}) \in I_q^* $.
Moreover, if we define the set of maps $E_q^*$ by $\{f_{-1}^* , f_0^* , f_1^* \}$ where
\begin{align*}
    \bqe{-1} &: \left[\frac{-1}{q-1} , \frac{2-q}{q(q-1)}\right] \rightarrow I_q^* ;
&\bqe{-1}(x) &= qx + 1, \\
    \bqe{0} &: \left[ \frac{-1}{q(q-1)}, \frac{1}{q(q-1)}\right] \rightarrow I_q^*;
&\bqe{0}(x) &= qx, \\
    \bqe{1} &: \left[ \frac{-(2-q)}{q(q-1)}, \frac{1}{q-1}\right] \rightarrow I_q^*;
&\bqe{1}(x) &= qx - 1,
\end{align*}
then Lemma \ref{lemma: relation between maps f and base q expansions} has the following natural generalisation.
However, we will not need this result until Subsection \ref{subsection: construction of A_q}.

\begin{lemma}
    \label{lemma: relation between maps and projections extension}
    For any $q \in (1,2)$, $\seqj{\delta}{1}{m} \in \{-1 , 0,1\}^*$, $\seqj{\epsilon}{1}{\infty} \in \{-1 ,0,1\}^\mathbb{N}$ we have,
    $$f_{\delta_1}^{*-1} \circ \cdots \circ f_{\delta_m}^{*-1} (\pi_q(\seqj{\epsilon}{1}{\infty})) = \pi_q( \seqj{\delta}{1}{m} \seqj{\epsilon}{1}{\infty}).$$
\end{lemma}

For any $k \in \mathbb{N}_{\geq 2}$, $q \in (1,2)$, define $g_{q,k} = f_1^{-(k-1)} \circ f_0^{-1}$, where the functions $f_i$ are implicitly dependent on $q$.
For brevity, whenever $m \geq 0$, we define 
$$N_k^m(q) = \bigcap_{i=0}^m g_{q,k}^i (\mathcal{U}_q + 1) \cap g_{q,k}^m(\mathcal{U}_q).$$

Proposition \ref{proposition: intersection gives nonempty B} shows that to prove $q \in \mathcal{B}_{m+2}$ it is sufficient to prove that $N_k^m(q)$ is nonempty.
\begin{proposition}
    \label{proposition: intersection gives nonempty B}
    Let $k \geq 2$, $q \in (G,2)$ and $m \geq 0$.
    If $x$ is such that $f_0(x) \in N_k^m(q)$, then $x \in \mathcal{U}_q^{(m+2)} \cap J_q$, and hence $q \in \mathcal{B}_{m+2}$.
\end{proposition}

To prove this proposition we will need the following definitions and lemmas.
Let $q \in (1,2)$, $x \in I_q$ and recall that $\Sigma_q(x)$ is the set of base $q$ expansions of $x$.
We define the orbit space of $x$ by
$$\orbs{q}{x} = \{ (f_{\epsilon_j})_{j=1}^\infty \in \{f_0 , f_1\}^\mathbb{N} : f_{\epsilon_k} \circ \cdots \circ f_{\epsilon_1}(x) \in I_q \ \mathrm{for \ all \ } k \in \mathbb{N} \}.$$
The following lemma was proved by the first author in \cite[Lemma 2.4]{Baker2014golden}.

\begin{lemma}
    \label{lemma: bijection between orbit space and sequence space}
    For any $x \in I_q$, $\Sigma_q(x)$ is in bijection with $\orbs{q}{x}$ via the map $\seqj{\epsilon}{1}{\infty} \rightarrow (f_{\epsilon_j})_{j=1}^\infty$.
\end{lemma}

This bijection allows one to see clearly how some $x \in I_q$ may have multiple base $q$ expansions.
The region $J_q = [\frac{1}{q} , \frac{1}{q(q-1)} ]$, known as the \textit{switch region} has the property that $x \in J_q$ if and only if $f_0(x) , f_1(x) \in I_q$.
Clearly, such an $x$ satisfies $|\orbs{q}{x}| \geq 2$ and hence $|\Sigma_q(x)| \geq 2$.
Let $x \in I_q \setminus J_q$, $y \in I_q$.
We say $x$ \textit{maps uniquely to} $y$ if there is a finite sequence of maps
$(f_{\epsilon_j})_{j=1}^m \in \{f_0 , f_1\}^*$ such that
$$f_{\epsilon_m} \circ \cdots \circ f_{\epsilon_1}(x) = y,$$
and
\begin{equation}
    \label{equation: maps uniquely to condition}
    f_{\epsilon_k} \circ \cdots \circ f_{\epsilon_1}(x) \in I_q \setminus J_q,
\end{equation}
for all $1 \leq k \leq m-1$.
Note that we allow $y \in J_q$.

\begin{lemma}
    \label{lemma: maps uniquely to condition}
    Let $q \in (1,2)$ and suppose $x \in I_q \setminus J_q$ maps uniquely to $y \in I_q$.
    Then $|\Sigma_q(x)| = |\Sigma_q(y)|$.
\end{lemma}


\begin{proof}
    Assuming the hypotheses of the lemma, we have $f_{\epsilon_k} \circ \cdots \circ f_{\epsilon_1}(x) \in I_q \setminus J_q$ for all $1 \leq k \leq m-1$, and $x \in I_q \setminus J_q$, where $(f_{\epsilon_j})_{j=1}^m$ exists by the assumption that $x$ maps uniquely to $y$.
    For any $z \in I_q \setminus J_q$, we can see by inspection of the domains of $f_0$ and $f_1$ that there is a unique element $\epsilon \in \{0,1\}$ such that $f_\epsilon(z) \in I_q$.
    Applying this to $x \in I_q \setminus J_q$ and to $f_{\epsilon_k} \circ \cdots \circ f_{\epsilon_1}(x)$ for every $1 \leq k \leq m-1$, we see that the sequence $(f_{\epsilon_j})_{j=1}^{m}$ is unique amongst those sequences which satisfy \eqref{equation: maps uniquely to condition} for all $1 \leq k \leq m-1$.
    Hence $(f_{\delta_j})_{j=1}^\infty \in \orbs{q}{x}$ if and only if $(f_{\delta_j})_{j=1}^\infty$ is of the form $(f_{\epsilon_j})_{j=1}^{m} (f_{\delta_j})_{j=m+1}^\infty$, where $(f_{\delta_j})_{j=m+1}^\infty \in \orbs{q}{f_{\epsilon_m} \circ \cdots \circ f_{\epsilon_1}(x)}$.
    Since $f_{\epsilon_m} \circ \cdots \circ f_{\epsilon_1}(x) = y$, we know $(f_{\delta_j})_{j=m+1}^\infty \in \orbs{q}{y}$.
    Therefore the elements of $\orbs{q}{x}$ are generated by taking the elements of $\orbs{q}{y}$ and adding a unique prefix of $(f_{\epsilon_j})_{j=1}^{m}$.
    Hence $|\orbs{q}{x}| = |\orbs{q}{y}|$ and by Lemma \ref{lemma: bijection between orbit space and sequence space} $|\Sigma_q(x)| = |\Sigma_q(y)|$.
\end{proof}

The first part of the following lemma is precisely stating the heuristic that if $q \in (G,2)$, then any orbit of a point under $E_q$ cannot remain in $J_q$ for more than one `step'.
The second part is a result touched upon in the introduction and is a consequence of Sidorov's work in \cite[Lemma 2.2]{sidorov2009expansions}, but for completion we prove it here anyway.
\begin{lemma}
    \label{lemma: switch region and B2 condition}
    Let $q \in (G,2)$.
    \begin{enumerate}[(a)]
        \item 
        \label{lemma item: maps out of switch}
        If $x \in J_q$ then for all $l \in \mathbb{N}$, the points $f_0^{-l}(x)$, $f_1^{-l}(x)$ map uniquely to $x$ via the sequences of maps given by $(f_0)_{j=1}^l$ and $(f_1)_{j=1}^l$ respectively.
        \item
        \label{lemma item: B2 condition}
        If $x \in I_q$ is such that $f_0(x) \in (\mathcal{U}_q + 1) \cap \mathcal{U}_q$ then $x \in \mathcal{U}_q^{(2)} \cap J_q$ and hence $q \in \mathcal{B}_2$.
    \end{enumerate}
\end{lemma}


\begin{proof}
    Part \ref{lemma item: maps out of switch}:
    Notice that $f_0^{-1}(x) = x/q$ is linear and strictly increasing for all $x \in I_q$ whenever $q \in (1,2)$.
    If $q \in (G,2)$ then $q^2 - q - 1 > 0$, so $\frac{q^{-2}}{q-1} < q^{-1}$ which is equivalent to $f_0^{-1}(\frac{1}{q(q-1)}) < \frac{1}{q}$.
    This means the image of the right endpoint of $J_q$ under $f_0^{-1}$ is less than the left endpoint of $J_q$.
    Hence, $f_0^{-1}(J_q) \cap J_q  = \emptyset$.
    Extending this we can see that $f_0^{-i}(\frac{1}{q(q-1)}) = \frac{1}{q^{i+1}(q-1)} < \frac{1}{q}$ for all $i \geq 1$ and hence $f_0^{-i}(J_q) \cap J_q = \emptyset$.
    Therefore, if $l \in \mathbb{N}$ and $x \in J_q$ is arbitrary, then $f_0^i \circ f_0^{-l}(x) \in f_0^{i-l}(J_q)$, hence $f_0^i \circ f_0^{-l}(x) \in I_q \setminus J_q$ for all $i = 0 , \ldots , l-1$.
    The case for $f_1$ is similar.

    Part \ref{lemma item: B2 condition}:
    Notice that for any $x \in J_q$, $f_0(x) - f_1(x) = 1$, so $f_0(x) - 1 = f_1(x)$.
    Therefore, if $f_0(x) \in (\mathcal{U}_q + 1)$ then $f_1(x) \in \mathcal{U}_q$ and since both $f_0(x), f_1(x) \in I_q$, we know that $x \in J_q$.
    If we also know that $f_0(x) \in \mathcal{U}_q$, then by Lemma \ref{lemma: bijection between orbit space and sequence space} $x \in \mathcal{U}_q^{(2)}$ and hence the claim.
\end{proof}

Equipped with Lemmas \ref{lemma: bijection between orbit space and sequence space}, \ref{lemma: maps uniquely to condition} and \ref{lemma: switch region and B2 condition} we are now able to prove Proposition \ref{proposition: intersection gives nonempty B}.

\begin{proof}[Proof of Proposition \ref{proposition: intersection gives nonempty B}]
    Fix any $q \in (1,2)$, $k \geq 2$ and proceed by induction on $m$.
    For the base case, let $m=0$ and suppose that $N_k^0(q) = (\mathcal{U}_q + 1) \cap \mathcal{U}_q \neq \emptyset$.
    Let $x \in I_q$ be such that $f_0(x) \in N_k^0(q)$.
    Then by Lemma \ref{lemma: switch region and B2 condition} \ref{lemma item: B2 condition}, $x \in \mathcal{U}_q^{(2)} \cap J_q$ and $q \in \mathcal{B}_2$.

    
    
    
    
    
    
    
    
    For the inductive step, notice that $N_k^{m+1}(q) = (\mathcal{U}_q + 1) \cap g_{q,k}(N_k^m(q))$.
    Assume the conclusion holds for $m$, namely that if $f_0(x) \in N_k^m(q)$ then $x \in \mathcal{U}_q^{(m+2)} \cap J_q$.
    Suppose that $x \in I_q$ is such that $f_0(x) \in N_k^{m+1}(q)$, then $f_1(x) \in \mathcal{U}_q$ because
    $f_0(x) \in (\mathcal{U}_q + 1)$.
    Since $f_0(x) \in g_{q,k}(N_k^m(q))$, we know that $f_0 \circ f_1^{k-1} \circ f_0(x) \in N_k^m(q)$.
    Therefore, by assumption, $f_1^{k-1} \circ f_0(x) \in \mathcal{U}_q^{(m+2)} \cap J_q$.
    Since $f_1^{k-1} \circ f_0(x) \in J_q$, Lemma \ref{lemma: switch region and B2 condition}\ref{lemma item: maps out of switch} implies that $f_0(x)$ maps uniquely to $f_1^{k-1} \circ f_0(x)$ and hence by Lemma \ref{lemma: maps uniquely to condition}, $f_0(x) \in \mathcal{U}_q^{(m+2)}$.
    We can conclude by Lemma \ref{lemma: bijection between orbit space and sequence space} that $x \in \mathcal{U}_q^{(m+3)}$ and since $f_0(x), f_1(x) \in I_q$, we know $x \in J_q$.
\end{proof}

Since $q > q_{k-1}$ implies that $\pi_q(S_{k-1}) \subset \mathcal{U}_q$ by Lemma \ref{lemma: containment in U_q} \ref{lemma item: U_q contains S_k}, we have the following lemma as an immediate corollary of Proposition \ref{proposition: intersection gives nonempty B}.

\begin{lemma}
\label{lemma: nonempty intersection}
    If $k \geq 2$, $m \geq 1$, $q > q_{k-1}$ and the intersection
    \begin{equation}
    \label{equation: intersection of S_{k-1}}
        \bigcap_{i=0}^m g_{q,k}^i(\pi_q(S_{k-1}) + 1) \cap g_{q,k}^m(\pi_q(S_{k-1})),
    \end{equation}
    is nonempty, then $q \in \mathcal{B}_{m+2}$.
\end{lemma}

We will see that it will be useful to consider subsets of the sets in \eqref{equation: intersection of S_{k-1}} rather than the sets in $N_k^m(q)$ because of our existing knowledge of the structure of the sets $\pi_q(S_{k-1})$.

\section{Proof of Theorem \ref{theorem: main theorem}}
\label{section: proof of main theorem}

Subsection \ref{subsection: Proof outline main theorem} serves to introduce our approach to proving Theorem \ref{theorem: main theorem}.

\subsection{Proof outline}
\label{subsection: Proof outline main theorem}

By Proposition \ref{proposition: intersection gives nonempty B}, for all $m \in \mathbb{N}$, if $N_k^m(q) \neq \emptyset$ then $q \in \mathcal{B}_{m+2}$, moreover if $f_0(x) \in N_k^m(q)$ then $x \in \mathcal{U}_q^{(m+2)} \cap J_q$, so a lower bound on the Hausdorff dimension of $N_k^m(q)$ provides us with a lower bound on the Hausdorff dimension of $\mathcal{U}_q^{(m+2)}$.
We seek to modify the sets in $N_k^m(q)$ in order to apply Theorem \ref{theorem: FandY}.
More precisely, for $i = 0, \ldots , m$, we seek subsets $P_i(q) \subset g_{q,k}^i(\mathcal{U}_q + 1)$ and $Q_m(q) \subset g_{q,k}^m(\mathcal{U}_q)$ such that Theorem \ref{theorem: FandY} can be applied to the collection of compact subsets of $\mathbb{R}$ given by $\{P_0(q) , \ldots , P_m(q) , Q_m(q)\}$.
For any given finite collection of sets $A_i \subset \mathbb{R}$, with arbitrary subsets $B_i \subset A_i$, it is obvious that if $\cap_i B_i \neq \emptyset$ then $\cap_i A_i \neq \emptyset$.
Hence, if 
$$M_k^m(q) = \bigcap_{i=0}^m P_i(q) \cap Q_m(q),$$
is nonempty, then $N_k^m(q) \neq \emptyset$ and $q \in \mathcal{B}_{m+2}$. 
Similarly, if $f_0(x) \in M_k^m(q)$ then $x \in \mathcal{U}_q^{(m+2)} \cap J_q$, so $\hdim(M_k^m(q)) \leq \hdim (\mathcal{U}_q^{(m+2)})$.
For $k$ sufficiently large, and $q > q_{k-1}$, each of the sets $P_0(q) , \ldots , P_m(q) , Q_m(q)$ is shown to have thickness at least $q^{k-4}$.
Let $B(q) = \cap_{i=0}^m \mathrm{conv}(P_i(q)) \cap \mathrm{conv}(Q_m(q))$ (see Figure \ref{figure: relative structure of sets}), which we show is nonempty and let 
$$\beta_q = \min \left\{ \frac{1}{4} , \frac{|B(q)|}{\max\{|P_0(q)| , \ldots , |P_m(q)| , |Q_m(q)|\}} \right\}.$$
$\beta_q$ corresponds to $\beta$ in Theorem \ref{theorem: FandY} except that we have dependency on $q$ here because the sets we are applying Theorem \ref{theorem: FandY} to depend on $q$.
To conclude Theorem \ref{theorem: main theorem} from Theorem \ref{theorem: FandY} it remains to find some $c \in (0,1)$ such that
\begin{equation}
\label{equation: applying FandY to P and Q}
  (m+2)(q^{k-4})^{-c} \leq \frac{1}{(432)^2}\beta_q^c(1-\beta_q^{1-c}),  
\end{equation}
holds for all $m \in \mathbb{N}$, $k \geq K_m$ and $|q-q_k| < q_k^{-(m+2)k-3}$.
Note that $c$ could depend on $q$ because the sets we are applying Theorem \ref{theorem: FandY} to depend on $q$.
However, since we eventually show that a constant value of $c$ satisfies \eqref{equation: applying FandY to P and Q} 
under the required conditions,
the standalone $c$ notation reflects the independence on $q$.
By requiring $|q - q_k| < q_k^{-(m+2)k-3}$ we achieve a uniform lower bound
on $\beta_q$ when $k \geq K_m$.
This bound on $\beta_q$ allows us to choose $c \in (0,1)$ to satisfy \eqref{equation: applying FandY to P and Q} as described above.
From here, Theorem \ref{theorem: FandY} tells us, under our assumptions $m \in \mathbb{N}$, $k \geq K_m$ and $|q-q_k| < q_k^{-(m+2)k-3}$, that $\hdim(\mathcal{U}_q^{(m+2)}) \geq \hdim(M_k^m(q)) \geq 1 - 1024(m+2)^{1/c}q^{k-4} > 0$.
Moreover, an immediate implication of this is that $\mathcal{U}_q^{(m+2)}$ is nonempty so $q \in \mathcal{B}_{m+2}$.
The majority of the work of the proof is in showing that the sets $P_0(q) , \ldots , P_m(q) , Q_m(q)$ satisfy the hypotheses of Theorem \ref{theorem: FandY}.
We will see that it is simpler to consider $P_i(q)$ as a subset of $g_{q,k}^i(\pi_q(S_{k-1}) + 1)$ for all $i = 0 , \ldots , m$ and $Q_m(q)$ as a subset of $g_{q,k}^m(\pi_q(S_{k-1}))$.
This is justified by Lemma \ref{lemma: nonempty intersection}.
In the following section we construct the subsets $P_i(q)$ of $g_k^i(\pi_q(S_{k-1}) + 1)$ for all $0 \leq i \leq m$ and $Q_m(q)$ of $g_k^m(\pi_q(S_{k-1}))$.

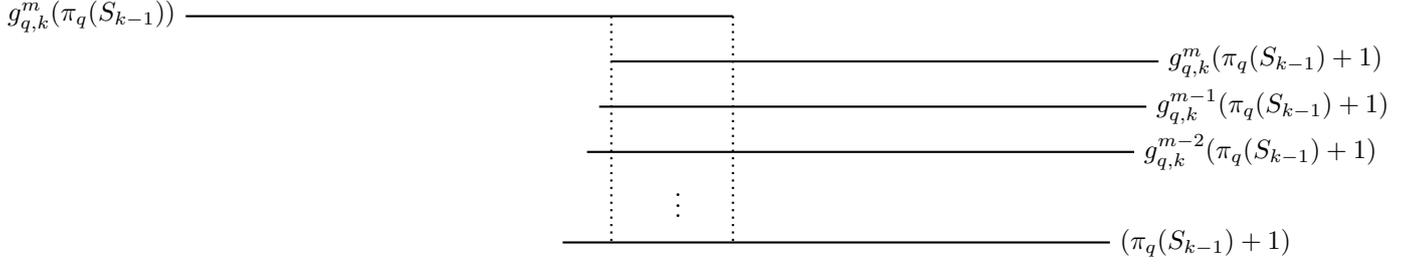
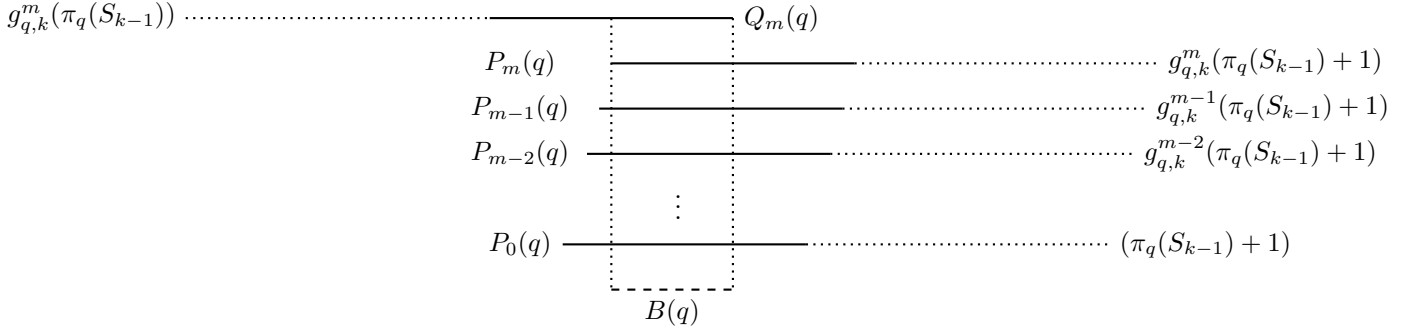
\begin{figure}
    \centering
    \begin{subfigure}[t]{\textwidth}
    \hspace{-0.7cm}
        \begin{tikzpicture}[yscale = 1.2, xscale = 0.8]
        \begin{scope}[thick]
            \draw (-7,0.5) -- (2,0.5);
            \draw (0,0) -- (9,0);
            \draw (-0.2,-0.5) -- (8.8,-0.5);
            \draw (-0.4,-1) -- (8.6,-1);
            \node at (1.1,-1.5) {$\vdots$};
            \draw (-0.8,-2) -- (8.2,-2);


            \draw[dotted] (0,0.5) -- (0,-2);
            \draw[dotted] (2,0.5) -- (2,-2);


            \node[left]   at (-7,0.5) {$g_{q,k}^m(\pi_q(S_{k-1}))$};
            \node [right] at (9,0) {$g_{q,k}^m(\pi_q(S_{k-1}) + 1)$};
            \node [right] at (8.8,-0.5) {$g_{q,k}^{m-1}(\pi_q(S_{k-1}) + 1)$};
            \node [right] at (8.6,-1) {$g_{q,k}^{m-2}(\pi_q(S_{k-1}) + 1)$};
            \node [right] at (8.2,-2) {$(\pi_q(S_{k-1}) + 1)$};

        \end{scope}
    \end{tikzpicture}
    \caption{The relative structure of the convex hulls of the sets in \eqref{equation: intersection of S_{k-1}} in the case when $|q - q_k| < q_k^{-2k-6}$ and $q < q_k$.}
        \label{figure: relative structure of convex hulls of large sets}
    \end{subfigure}
    \begin{subfigure}[t]{\textwidth}
    \hspace{-0.7cm}
        \begin{tikzpicture}[yscale = 1.2, xscale = 0.8]
        \begin{scope}[thick]
            \draw (-2,0.5) -- (2,0.5);
            \draw (0,0) -- (4,0);
            \draw (-0.2,-0.5) -- (3.8,-0.5);
            \draw (-0.4,-1) -- (3.6,-1);
            \node at (1.1,-1.5) {$\vdots$};
            \draw (-0.8,-2) -- (3.2,-2);

            \draw[dotted] (-7,0.5) -- (-2,0.5);
            \draw[dotted] (4,0) -- (9,0);
            \draw[dotted] (3.8,-0.5) -- (8.8,-0.5);
            \draw[dotted] (3.6,-1) -- (8.6,-1);
            \draw[dotted] (3.2 ,-2) -- (8.2,-2);

            \draw[dotted] (0,0.5) -- (0,-2.5);
            \draw[dotted] (2,0.5) -- (2,-2.5);
            \draw[dashed] (0,-2.5) -- (2,-2.5);

            \node at (2.8,0.5) {$Q_m(q)$};
            \node at (-1.5,0) {$P_m(q)$};
            \node at (-1.5,-0.5) {$P_{m-1}(q)$};
            \node at (-1.5, -1) {$P_{m-2}(q)$};
            \node at (-1.5,-2) {$P_0(q)$};

            \node[left]   at (-7,0.5) {$g_{q,k}^m(\pi_q(S_{k-1}))$};
            \node [right] at (9,0) {$g_{q,k}^m(\pi_q(S_{k-1}) + 1)$};
            \node [right] at (8.8,-0.5) {$g_{q,k}^{m-1}(\pi_q(S_{k-1}) + 1)$};
            \node [right] at (8.6,-1) {$g_{q,k}^{m-2}(\pi_q(S_{k-1}) + 1)$};
            \node [right] at (8.2,-2) {$(\pi_q(S_{k-1}) + 1)$};

            \node[below] at (1,-2.5) {$B(q)$};
        \end{scope}
    \end{tikzpicture}
    \caption{For $q$ as in Figure \ref{figure: relative structure of convex hulls of large sets}, the figure shows the process of taking subsets given by $\{P_0(q) , \ldots , P_m(q) , Q_m(q)\}$, and their overlap given by $B(q)$.
    For clarity, only the convex hulls of the $P_0(q) , \ldots , P_m(q) , Q_m(q)$ sets are shown.}
        \label{figure: relative structure of convex hulls of small sets}
    \end{subfigure}
    \caption{The process of taking subsets of affine images of $\pi_q(S_{k-1})$ in order to bound $\beta_q$.
    For $|q - q_k| < q_k^{-2k-6}$, the relative structure of the sets in all cases where $q < q_k$, $q = q_k$ and $q > q_k$ are shown in Figure \ref{figure: relative structure of P and Q sets}.}
    \label{figure: relative structure of sets}
\end{figure}

\subsection{Construction of $P_i(q)$ and $Q_m(q)$}

Fix $m \in \mathbb{N}$, $k \geq 3$ and $q > q_{k-1}$.
Recall from Subsection \ref{subsection: thickness and interleaving} that any compact subset of $\mathbb{R}$ admits a construction as a complement of at most countably many gaps.
Let $G(q)$ be the gap\footnote{Recall that we know this is a gap of $\pi_q(S_{k-1})$ by Lemma \ref{lemma: containment in U_q}\ref{lemma item: containment in U_q Cantor set}.} of $\pi_q(S_{k-1})$ given by the open interval 
$$G(q) = (\pi_q(0^{k-3} (01^{k-1})^\infty) , \pi_q(0^{k-3} (10^{k-1})^\infty)).$$
Similarly let $H(q)$ be the gap of $\pi_q(S_{k-1})$ given by the open interval 
$$H(q) = (\pi_q(1^{k-3}(01^{k-1})^\infty) , \pi_q(1^{k-3}(10^{k-1})^\infty)).$$
Let $0 \leq i \leq m$ and define $P_i(q)$ to be the subset of $g_{q,k}^i(\pi_q(S_{k-1}) + 1)$ to the left of the gap $G_i(q)$ where
$$G_i(q)     = g_{q,k}^i \left( q^{-(m-i)k}G(q) + 1 \right),$$
that is,
$$P_i(q) = g_{q,k}^i(\pi_q(S_{k-1}) + 1) \cap [g_{q,k}^i(1) , g_{q,k}^i(1 + q^{-(m-i)k}\pi_q(0^{k-3}(10^{k-1})^\infty))]. $$
Similarly, define $Q_m(q)$ to be the subset of $g_{q,k}^m(\pi_q(S_{k-1}))$ to the right of the gap
$H_m(q)$ where
$$H_m(q)     = g_{q,k}^m \left( H(q) \right),$$
that is,
$$Q_m(q) = g_{q,k}^m(\pi_q(S_{k-1})) \cap [g_{q,k}^m(\pi_q(1^{k-3}(10^{k-1})^\infty)) , g_{q,k}^m(\pi_q(1^\infty))].$$

Since 
$$(\pi_q(0^{(m-i)k} 0^{k-3} (01^{k-1})^\infty) , \pi_q(0^{(m-i)k} 0^{k-3} (10^{k-1})^\infty))$$
is also a gap of $\pi_q(S_{k-1})$ for all $0 \leq i \leq m$ by Lemma \ref{lemma: containment in U_q} \ref{lemma item: containment in U_q Cantor set}, and we can write the gap $G_i(q)$ as
$$g_{q,k}^i\left( \pi_q(0^{(m-i)k} 0^{k-3} (01^{k-1})^\infty) + 1 , \pi_q(0^{(m-i)k} 0^{k-3} (10^{k-1})^\infty) + 1 \right),$$
we know that $G_i(q)$ is a gap of $g_{q,k}^i(\pi_q(S_{k-1})+ 1)$.
Similarly, $H_m(q)$ is a gap of $g_{q,k}^m(\pi_q(S_{k-1}))$.
By Lemma \ref{lemma: thickness of S_k}, we know that $\tau(\pi_q(S_{k-1})) > q^{k-4}$.
To apply Theorem \ref{theorem: FandY} we need to make sure the process of taking the subsets $P_0(q) , \ldots , P_m(q) , Q_m(q)$ described above preserves this lower bound on the thickness.
This is the content of Lemma \ref{lemma: bound thickness of P_i below}.
To prove this, we will need the following definition and a lemma which is another consequence of \cite[Lemma 4]{GlenSid2001}.
We say that $\seqj{\epsilon}{1}{\infty}$ is \textit{lexicographically less than} $\seqj{\epsilon'}{1}{\infty}$ if $\epsilon_i < \epsilon'_i$ where $i = \min\{j : \epsilon_j \neq \epsilon'_j\}$.
In this case we write $\seqj{\epsilon}{1}{\infty} \prec \seqj{\epsilon'}{1}{\infty}$.

\begin{lemma}
\label{lemma: properties of gaps}
Let $k \geq 2$, $q \in (q_k,2)$ and let $ \delta_1, \delta_2 \in \{0,1\}^*$ be finite sequences which avoid $(01^k)$ and $(10^k)$.\
For $i \in \{1,2\}$, let $G_i(q)$ be the gap of $\pi_q(S_k)$ given by the open interval
$$\left(\pi_q(\delta_i (01^{k-1})^\infty) , \pi_q(\delta_i (10^{k-1})^\infty)\right).$$
    \begin{enumerate}[(a)]
        \item 
        If $|\delta_1| < |\delta_2|$ then $|G_1| > |G_2|$.
        \item 
        If $|\delta_1| = |\delta_2|$ and $\delta_1 \prec \delta_2$ then the right endpoint of $G_1$ is less than the left endpoint of $G_2$.
    \end{enumerate}
\end{lemma}

\begin{lemma}
\label{lemma: bound thickness of P_i below}
    For $k \geq 3$ and any $0 \leq i \leq m$, $\tau(P_i(q)) \geq \tau (\pi_q(S_{k-1}))$ and $\tau(Q_m(q)) \geq \tau(\pi_q(S_{k-1}))$.
\end{lemma}

\begin{proof}
    Let $k \geq 3$, $q > q_{k-1}$ and let $P^*(q)$ be the subset of $\pi_q(S_{k-1})$ to the left of the gap given by $G(q)$.
    That is
    $$P^*(q) = \pi_q(S_{k-1}) \cap [0, \pi_q(0^{k-3}(01^{k-1})^\infty)].$$
    We require $k \geq 3$ in order for the gap $G(q)$ to be well-defined, and we require $q > q_{k-1}$ to guarantee that $\pi_q(S_{k-1})$ is a Cantor set by Lemma \ref{lemma: containment in U_q} \ref{lemma item: containment in U_q Cantor set}.
    Thickness is invariant under affine maps, and for each $0 \leq i \leq m$, $P_i(q)$ is the image under an affine map of $P^*(q)$.
    Hence, to prove the first part it is sufficient to show that $\tau(P^*(q)) \geq \tau(\pi_q(S_{k-1}))$.
    The result for $\tau(Q_m(q))$ can be proved in a similar way, and the lemma follows.

    Using Lemma \ref{lemma: properties of gaps}, since $0^{k-3}$ is the lexicographically smallest sequence of length $k-3$, we know that all gaps of $\pi_q(S_{k-1})$ contained in the interval
    $[0, \pi_q(0^{k-3} (01^{k-1})^\infty)]$ must be smaller than $G(q)$.
    Hence $\mathrm{conv}(P^*(q))$ is a bridge of $\pi_q(S_{k-1})$ 
    so all bounded gaps of $\pi_q(S_{k-1})$ are either contained in $\mathrm{conv}(P^*(q))$ or contained in the complement of $\mathrm{conv}(P^*(q))$.
    When we evaluate the thickness of $P^*(q)$, we can observe that the set of bounded gaps over which the infimum is taken is a subset of the set of bounded gaps of $\pi_q(S_{k-1})$.
    Moreover, for any gap $G'$ contained in $\mathrm{conv}(P^*(q))$, the left and right bridges of $G'$ will be the same regardless of whether we consider $G'$ to be a gap of $P^*(q)$ or a gap of $\pi_q(S_{k-1})$.   
    Let $\mathcal{G}$ be the set of bounded gaps of $\pi_q(S_{k-1})$ and let $\mathcal{G}_{P^*}$ be the set of bounded gaps of $P^*(q)$.
    With the above observation in mind, and the fact that $\mathcal{G}_{P^*} \subset \mathcal{G}$, we see that,
    $$\tau(P^*(q)) = \inf \left\{ \min \left\{ \frac{|L_n|}{|G_n|} , \frac{|R_n|}{|G_n|} \right\} : G_n \in \mathcal{G}_{P^*} \right\} \geq \inf \left\{ \min \left\{ \frac{|L_n|}{|G_n|} , \frac{|R_n|}{|G_n|} \right\} : G_n \in \mathcal{G} \right\} = \tau(\pi_q(S_{k-1})).$$
\end{proof}

We will only need to apply Lemma \ref{lemma: bound thickness of P_i below} in the case $k \geq 5$ because this is where the bound $\tau(\pi_q(S_{k-1})) \geq q^{k-4}$ provided by Lemma \ref{lemma: thickness of S_k} applies.


\subsection{The relative structure of the sets $P_i(q)$ and $Q_m(q)$}
\label{subsection: relative structure of sets P_i and Q_m}

Recall that our plan is to apply Theorem \ref{theorem: FandY} to the collection $\{P_0(q) , \ldots , P_m(q) , Q_m(q) \}$ where $B(q)$ is the intersection of their convex hulls.
In order for the convex hulls of the sets $P_0(q) , \ldots , P_m(q) , Q_m(q)$ to overlap in a sufficiently large interval, i.e. for $|B(q)|$ to be sufficiently large, we need to impose stronger conditions on $q$.
This is done in the following lemmas.
Note that the projection of any sequence - besides $0^\infty$ - decreases as $q$ increases.
This is the content of the following lemma.

\begin{lemma}
    \label{lemma: projections decrease with q}
    For any sequence $\seqj{\delta}{1}{\infty} \in \{0,1\}^\mathbb{N} \setminus \{0^\infty \}$, and any $q,q' \in (1,2)$,
    $$q < q' \iff \pi_{q'}(\seqj{\delta}{1}{\infty}) < \pi_q(\seqj{\delta}{1}{\infty}).$$
\end{lemma}

Given some $q \in (1,2)$, $k \geq 2$, define $\epsilon_q \in \mathbb{R}$ by
\begin{equation}
    \label{equation: definition of epsilon_q}
    1 = \pi_q((1^{k-1}0)^\infty) + \epsilon_q.
\end{equation}

We introduce $\epsilon_q$ as a means to bound $|B(q)|$ and hence also $\beta_q$ from below.
Recall from Subsection \ref{subsection: Proof outline main theorem} that a lower bound on $\beta_q$ allows us to prove \eqref{equation: applying FandY to P and Q} holds for some $c \in (0,1)$ and ultimately apply Theorem \ref{theorem: FandY} to prove Theorem \ref{theorem: main theorem}.
We will always be dealing with $q$ in a small neighbourhood of $q_k$ for some fixed $k$ so the implicit dependence of $\epsilon_q$ on $k$ is suppressed in the notation.
Since $1 = \pi_{q_k}((1^{k-1}0)^\infty)$, Lemma \ref{lemma: projections decrease with q} tells us that $\epsilon_q < 0$ if $q < q_k$, $\epsilon_q = 0$ if $q = q_k$ and $\epsilon_q > 0$ if $q > q_k$.
The purpose of the following two lemmas is to find an upper bound on $|\epsilon_q|$ given an upper bound on $|q - q_k|$.

\begin{lemma}
    \label{lemma: expansion of 1 has certain prefix}
    Let $m \in \mathbb{N}$, $k \geq 4$.
    If $|q - q_k| < q_k^{-(m+1)k-3}$ then there is some sequence $\seqj{c}{1}{\infty} \in \{0,1\}^\mathbb{N}$ such that $1 = \pi_q(\seqj{c}{1}{\infty})$ and $\seqj{c}{1}{\infty} = (1^{k-1}0)^m \seqj{c}{mk+1}{\infty}$.
\end{lemma}

\begin{proof}
    Let $m \in \mathbb{N}$, $k \geq 4$.
    Let $q_- \in (1,2)$ be such that $1 = \pi_{q_-}((1^{k-1}0)^m 0^\infty)$ and let $q_+ \in (1,2)$ be such that $1 = \pi_{q_+}((1^{k-1}0)^m 1^\infty)$.
    By Lemma \ref{lemma: projections decrease with q}, $q \in (q_- , q_+)$ is equivalent to
    \begin{equation}
    \label{equation: q- q+ inequality}
\pi_q((1^{k-1}0)^m 0^\infty) < \pi_{q_-}((1^{k-1}0)^m 0^\infty) = 1 = \pi_{q_+}((1^{k-1}0)^m 1^\infty) < \pi_q((1^{k-1}0)^m 1^\infty).
    \end{equation}
    This means that, 
    $$1 \in (\pi_q((1^{k-1}0)^m 0^\infty) , \pi_q((1^{k-1}0)^m 1^\infty)),$$ 
    that is, there must exist some sequence $\seqj{c}{1}{\infty} \in \{0,1\}^\mathbb{N}$ such that $1 = \pi_q(\seqj{c}{1}{\infty})$ where $\seqj{c}{1}{\infty} = (1^{k-1}0)^m \seqj{c}{km+1}{\infty}$.
    Therefore, to prove the lemma, it suffices to show that if $|q - q_k| < q_k^{-(m+1)k-3}$ then $q \in (q_- , q_+)$.
    
    If $q = q_k$ then we know that $1 = \pi_q((1^{k-1}0)^\infty)$ which immediately gives $\seqj{c}{1}{\infty} = (1^{k-1}0)^\infty$ and also $q_k \in (q_- , q_+)$.
    Assume that $m \in \mathbb{N}$, $k \geq 4$ and $|q - q_k| < q_k^{-(m+1)k-3}$.
    Since $1 = \pi_{q_k}((1^{k-1}0)^\infty)$, \eqref{equation: q- q+ inequality} becomes
    \begin{equation}
        \label{equation: q_k inequality}
        \pi_q((1^{k-1}0)^m 0^\infty) < \pi_{q_k}((1^{k-1}0)^\infty) < \pi_q((1^{k-1}0)^m 1^\infty),
    \end{equation}
    and so this is also equivalent to $q \in (q_- , q_+)$.
    If $q < q_k$ then by exchanging $\pi_q((1^{k-1}0)^m 1^\infty)$ for the smaller expression $\pi_q((1^{k-1}0)^\infty)$, the second inequality of \eqref{equation: q_k inequality} is implied by
    $$\pi_{q_k}((1^{k-1}0)^\infty) < \pi_q((1^{k-1}0)^\infty),$$
    which is itself a direct consequence of Lemma \ref{lemma: projections decrease with q}.
    If $q > q_k$ then the first inequality of \eqref{equation: q_k inequality} holds by similar reasoning.
    Therefore, it suffices to show
    \begin{enumerate}
        \item if $q < q_k$ then $\pi_q((1^{k-1}0)^m 0^\infty) < \pi_{q_k}((1^{k-1}0)^\infty)$, 
    \end{enumerate}
        and
    \begin{enumerate}
        \item [2.] if $q > q_k$ then $\pi_{q_k}((1^{k-1}0)^\infty) < \pi_q((1^{k-1}0)^m 1^\infty)$.
    \end{enumerate}

    \textbf{Case 1}.
    
    Let $q < q_k$.
    Using the fact that $\pi_q((1^{k-1}0)^m 0^\infty) = (1-q^{-km})\pi_q((1^{k-1}0)^\infty)$, the inequality $\pi_q((1^{k-1}0)^m 0^\infty) < \pi_{q_k}((1^{k-1}0)^\infty)$ becomes
    \begin{align}
    \label{equation: q < q_k intermediate inequality}
     q^{-km} \pi_q((1^{k-1}0)^\infty) 
        &> \pi_q((1^{k-1}0)^\infty) - \pi_{q_k}((1^{k-1}0)^\infty), \nonumber \\
        &= \left( \frac{1}{q-1} - \frac{1}{q^k-1}\right) - \left( \frac{1}{q_k-1} - \frac{1}{q_k^k-1}\right), \nonumber \\
        &= \left( \frac{1}{q-1} - \frac{1}{q_k-1} \right) - \left( \frac{1}{q^k-1} - \frac{1}{q_k^k-1}\right).
    \end{align}
    We show this inequality is a consequence of our assumption that $|q - q_k| < q_k^{-(m+1)k-3}$.
    The second term in brackets in \eqref{equation: q < q_k intermediate inequality} is positive and the first term in brackets equals $\frac{q_k - q}{(q-1)(q_k-1)}$.
    Using Lemma \ref{lemma: projections decrease with q} again, $\pi_q((1^{k-1}0)^\infty) > \pi_{q_k}((1^{k-1}0)^\infty) = 1$ so \eqref{equation: q < q_k intermediate inequality} is implied by
    \begin{equation}
    \label{equation: q- q+ intermediate inequality}
    q^{-km} > \frac{q_k-q}{(q-1)(q_k-1)}.
    \end{equation}
    With $k \geq 4$, we know\footnote{Since $q_4 \approx 1.9276$ (by direct computation), the smallest possible value of $q$ is $q_4 - q_4^{-11}$, so this is a very loose lower bound.} $q > 15/8$ so $(q-1)(q_k-1) > 49/64 > q^{-1}$ and hence \eqref{equation: q- q+ intermediate inequality} is implied by
    $$q_k - q < q_k^{-km-1},$$
    which is definitely true if $q_k - q < q_k^{-(m+1)k-3}$.
    This completes the first case.

    \textbf{Case 2}

    Let $q > q_k$.
    We have the following chain of equivalences from the inequality we set out to prove:
    \begin{align}
    \label{equation: q > q_k intermediate inequality}
        & & \pi_{q_k}((1^{k-1}0)^\infty) & <  \pi_q((1^{k-1}0)^m 1^\infty), \nonumber \\
        & \Leftrightarrow & \pi_{q_k}((1^{k-1}0)^\infty) - \pi_q((1^{k-1}0)^\infty) & < q^{-km}\pi_q((0^{k-1}1)^\infty), \nonumber \\
        & \Leftrightarrow & \left( \frac{1}{q_k-1} - \frac{1}{q-1}\right) - \left( \frac{1}{q_k^k - 1} - \frac{1}{q^k-1} \right) & < q^{-km} \frac{1}{q^k-1}.    
    \end{align}
    The first equivalence makes use of the following equalities
    $$\pi_q((1^{k-1}0)^m 1^\infty) = \pi_q((1^{k-1}0)^\infty) + \pi_q(0^{km} (0^{k-1}1)^\infty) = \pi_q((1^{k-1}0)^\infty) + q^{-km}\pi_q((0^{k-1}1)^\infty),$$
    while the second is simply rewriting the $\pi_q$ expressions using the standard formulas for geometric series, and rearranging them.
    As in the previous case, the second term in brackets on the left hand side of \eqref{equation: q > q_k intermediate inequality} is positive, which means this inequality is implied by
    $$\frac{q-q_k}{(q_k-1)(q-1)} < \frac{q^{-km}}{q^k-1},$$
    which we rearrange to
    $$q-q_k < q^{-km}\frac{(q_k-1)(q-1)}{q^k-1}.$$
    Again, we aim to show this is a consequence of our assumptions on $q$.
    Using $k \geq 4$, it can be checked that $q^{-km} > q_k^{-km-1}$ and $\frac{(q_k-1)(q-1)}{q^k-1} > q_k^{-k-2}$, so the above inequality is implied by
    $$ q-q_k < q^{-(m+1)k-3},$$
    which is obviously a restriction of $|q - q_k| < q^{-(m+1)k-3}$ to the case $q > q_k$.
    Therefore the proof of the second case is complete.
    Hence we have proved the conclusion in both cases and the lemma holds.
\end{proof}

Recall the definition of $\epsilon_q$ from \eqref{equation: definition of epsilon_q} above.

\begin{lemma}
\label{lemma: bounding epsilon in terms of q}
    Let $m \in \mathbb{N}$, $k \geq 4$
    and $|q - q_k| < q_k^{-(m+2)k-3}$. 
    Then $-q_k^{-(m+1)k+1} < \epsilon_q < q_k^{-(m+2)k + 1}$.
\end{lemma} 

\begin{proof}
    Assuming the hypotheses of the lemma, let $\seqj{c}{1}{\infty} \in \{0,1\}^\mathbb{N}$ be of the form $(1^{k-1}0)^{m+1}\seqj{c}{k(m+1)+1}{\infty}$ where $\pi_q(\seqj{c}{1}{\infty}) = 1$, which we know exists
    by Lemma \ref{lemma: expansion of 1 has certain prefix}.
    Then,
    \begin{equation}
    \label{equation: epsilon inequality}
    \pi_q((1^{k-1}0)^{m+1} 0^\infty) \leq \pi_q((1^{k-1}0)^\infty) + \epsilon_q \leq \pi_q((1^{k-1}0)^{m+1} 1^\infty).
    \end{equation}
    If $q = q_k$ then $\epsilon_q = 0$ so we consider the remaining two cases separately.

    \textbf{Case 1: $q < q_k$.}
    If $q < q_k$ then $\epsilon_q < 0$ and the first part of \eqref{equation: epsilon inequality} is equivalent to
    $$\epsilon_q    \geq (1-q^{-k(m+1)})\pi_q((1^{k-1}0)^\infty) - \pi_q((1^{k-1}0)^\infty) = -q^{-k(m+1)}(1-\epsilon_q),$$
    so 
    $$\epsilon_q \geq -\frac{q^{-k(m+1)}}{1-q^{-k(m+1)}}.$$
    From here, using $k \geq 4$, it can be checked that $\epsilon_q > -q_k^{-k(m+1)+1}$.
    This is intuitive given the factor $\frac{1}{1-q^{-km}}$ is very close to $1$ and $q$ is very close to $q_k$.
    
    \textbf{Case 2: $q > q_k$.}
    If $q > q_k$ then $\epsilon_q > 0$ and the second part of \eqref{equation: epsilon inequality} is equivalent to
    $$\pi_q((1^{k-1}0)^\infty) + \epsilon_q \leq \pi_q((1^{k-1}0)^{m+1} 1^\infty),$$
    which is equivalent to
    $$\epsilon_q \leq \frac{q^{-k(m+1)}}{q^k-1}.$$
    It can be checked that this implies that
    $$\epsilon_q < q_k^{-(m+2)k + 1}.$$
    
    Combining the two cases, we see that $-q_k^{-(m+1)k+1} < \epsilon_q < q_k^{-(m+2)k + 1}$.
\end{proof}

For $q \in (1,2)$ and $k \geq 2$, the map $g_{q,k} = f_1^{-(k-1)} \circ f_0^{-1}$ has a natural symbolic interpretation.
Let $\seqj{\delta}{1}{\infty} \in \{0,1\}^\mathbb{N}$ be a base $q$ expansion for some point $y \in I_q$, then an application of Lemma \ref{lemma: relation between maps f and base q expansions} is that
\begin{equation}
    \label{equation: g_qk prefixes by a sequence}
    g_{q,k}(\pi_q(\seqj{\delta}{1}{\infty})) = \pi_q((1^{k-1}0)\seqj{\delta}{1}{\infty}).
\end{equation}
That is, the point $g_{q,k}(y)$ has a base $q$ expansion given by $(1^{k-1}0)\seqj{\delta}{1}{\infty}$ where $\seqj{\delta}{1}{\infty}$ is a base $q$ expansion of $y$.
In other words, $g_{q,k}$ prefixes any base $q$ expansion of $y$ by $(1^{k-1}0)$.
By inspection $g_{q,k}$ is an affine transformation whose linear part is a contraction by $q^k$ and whose translation is a shift of $\pi_q(1^{k-1}0^\infty)$.
It is straightforward to check that the fixed point of $g_{q,k}$ is $\pi_q((1^{k-1}0)^\infty)$.
For any $k \geq 2$ and any $x$ such that $\pi_q((1^{k-1}0)^\infty) + x \in I_q$,
\begin{equation}
    \label{equation: affine transformation fixed point}
    g_{q,k}(\pi_q((1^{k-1}0)^\infty) + x) = \pi_q((1^{k-1}0)^\infty) + q^{-k}x.
\end{equation}
We will also use the fact that since $g_{q,k}$ is an affine transformation, 
\begin{equation}
    \label{equation: affine transformation translation}
    g_{q,k}(x) - g_{q,k}(z) = q^{-k}(x-z)
\end{equation}
whenever $x,z \in I_q$.

The following lemmas allows us to determine the interval $B(q)$ as $q$ varies relative to $q_k$.
The proof involves useful bounds on the diameters of the $P_i(q)$ and $Q_m(q)$ sets which are needed in order to bound $\beta_q$ below in Subsection \ref{subsection: bounding beta below}.
We denote by $L(X)$ and $R(X)$ the left and right endpoints of a compact subset $X \subset \mathbb{R}$.
Figure \ref{figure: relative structure of P and Q sets} provides a useful reference for Lemma \ref{lemma: layout of endpoint of P and Q}.

\begin{lemma}
    \label{lemma: P and Q sets have a common diameter}
    For any $m \in \mathbb{N}$, $k \geq 3$ and $q > q_{k-1}$, 
        $$|P_0(q)| = \cdots = |P_m(q)| = |Q_m(q)| = q^{-(m+1)k + 2} \pi_q((1^{k-1}0)^\infty).$$
\end{lemma}

\begin{proof}    
    Recall that for $k \geq 3$, $q > q_{k-1}$
    $$G_i(q)     = g_{q,k}^i \left( q^{-(m-i)k}G(q) + 1 \right),$$
    where
    $$G(q) = (\pi_q(0^{k-3} (01^{k-1})^\infty) , \pi_q(0^{k-3} (10^{k-1})^\infty)),$$
    and that $P_i(q)$ is the subset of $g_{q,k}^i(\pi_q(S_{k-1}) + 1)$ to the left of $G_i(q)$.
    So $P_i(q)$ has the same left endpoint as $g_{q,k}^i(\pi_q(S_{k-1}) + 1)$, namely $g_{q,k}^i(\pi_q(0^\infty) + 1) = g_{q,k}^i(1)$.
    Observe that, using \eqref{equation: affine transformation translation}, for any $0 \leq i \leq m$,
    \begin{align}
        |P_i(q)|    &= L(G_i(q)) - L(P_i(q)), \nonumber \\
                    &= g_{q,k}^i(\pi_q(0^{(m-i)k} 0^{k-3} (01^{k-1})^\infty) + 1) - g_{q,k}^i(1), \nonumber \\
                    &= q^{-km}(\pi_q(0^{k-3} (01^{k-1})^\infty)), \nonumber \\
                    &= q^{-(m+1)k + 2} \pi_q((1^{k-1}0)^\infty), \nonumber 
        \end{align}    
    which we note is independent of $i$.
    Recall also that
    $$H_m(q)     = g_{q,k}^m \left( H(q) \right),$$
    where 
    $$H(q) = (\pi_q(1^{k-3}(01^{k-1})^\infty) , \pi_q(1^{k-3}(10^{k-1})^\infty)),$$
    and that $Q_m(q)$ is the subset of $g_{q,k}^m(\pi_q(S_{k-1}))$ to the right of the gap
    $H_m(q)$.
    Using \eqref{equation: affine transformation translation} again, we see that
        \begin{align*}
            |Q_m(q)|   &= R(Q_m(q)) - R(H_m(q)),\\
                    &= g_{q,k}^m(\pi_q(1^\infty)) - g_{q,k}^m(\pi_q(1^{k-3} (10^{k-1})^\infty)),\\
                    &= q^{-km}(\pi_q(0^{k-3} (01^{k-1})^\infty)),\\
                    &= |P_i(q)|,
        \end{align*}    
    for all $0 \leq i \leq m$.
\end{proof}

By Lemma \ref{lemma: P and Q sets have a common diameter}, we know that $|P_0(q)| = \cdots = |P_m(q)| = |Q_m(q)|$.
We denote this value by $D_q$.

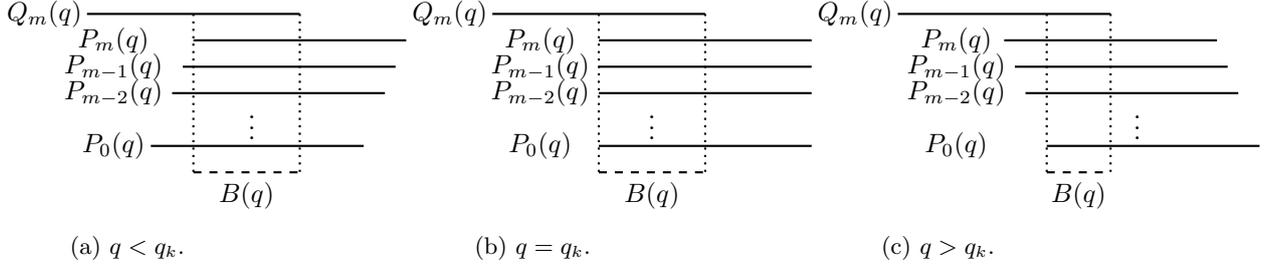
\begin{figure}[H]
            \begin{subfigure}[t]{0.2\textwidth}
        \vskip 0pt
     \begin{tikzpicture}[scale = 0.7]
        \begin{scope}[thick]
            \draw (-2,0.5) -- (2,0.5);
            \draw (0,0) -- (4,0);
            \draw (-0.2,-0.5) -- (3.8,-0.5);
            \draw (-0.4,-1) -- (3.6,-1);
            \node at (1.1,-1.5) {$\vdots$};
            \draw (-0.8,-2) -- (3.2,-2);

            \draw[dotted] (0,0.5) -- (0,-2.5);
            \draw[dotted] (2,0.5) -- (2,-2.5);
            \draw[dashed] (0,-2.5) -- (2,-2.5);

            \node at (-2.8,0.5) {$Q_m(q)$};
            \node at (-1.5,0) {$P_m(q)$};
            \node at (-1.5,-0.5) {$P_{m-1}(q)$};
            \node at (-1.5, -1) {$P_{m-2}(q)$};
            \node at (-1.5,-2) {$P_0(q)$};

            \node[below] at (1,-2.5) {$B(q)$};
        \end{scope}
    \end{tikzpicture}
    \caption{$q < q_k$.}
    \label{figure: structure of intersecting sets (-ve delta)}
        \end{subfigure}
        \hspace{17mm}
        \begin{subfigure}[t]{0.2\textwidth}
        \vskip 0pt
\begin{tikzpicture}[scale = 0.7]
        \begin{scope}[thick]
            \draw (-2,0.5) -- (2,0.5);
            \draw (0,0) -- (4,0);
            \draw (0,-0.5) -- (4,-0.5);
            \draw (0,-1) -- (4,-1);
            \node at (1,-1.5) {$\vdots$};
            \draw (0,-2) -- (4,-2);

            \draw[dotted] (0,0.5) -- (0,-2.5);
            \draw[dotted] (2,0.5) -- (2,-2.5);
            \draw[dashed] (0,-2.5) -- (2,-2.5);

            \node at (-2.8,0.5) {$Q_m(q)$};
            \node at (-1.1,0) {$P_m(q)$};
            \node at (-1.1,-0.5) {$P_{m-1}(q)$};
            \node at (-1.1, -1) {$P_{m-2}(q)$};
            \node at (-1.1,-2) {$P_0(q)$};

            \node[below] at (1,-2.5) {$B(q)$};
        \end{scope}
    \end{tikzpicture}
    \caption{$q = q_k$.}
    \label{figure: structure of intersecting sets (+ve delta)}
        \end{subfigure}
        \hspace{17mm}
 \begin{subfigure}[t]{0.2\textwidth}
        \vskip 0pt
\begin{tikzpicture}[scale = 0.7]
        \begin{scope}[thick]
            \draw (-2,0.5) -- (2,0.5);
            \draw (0,0) -- (4,0);
            \draw (0.2,-0.5) -- (4.2,-0.5);
            \draw (0.4,-1) -- (4.4,-1);
            \node at (2.5,-1.5) {$\vdots$};
            \draw (0.8,-2) -- (4.8,-2);

            \draw[dotted] (0.8,0.5) -- (0.8,-2.5);
            \draw[dotted] (2,0.5) -- (2,-2.5);
            \draw[dashed] (0.8,-2.5) -- (2,-2.5);

            \node at (-2.8,0.5) {$Q_m(q)$};
            \node at (-0.9,0) {$P_m(q)$};
            \node at (-0.9,-0.5) {$P_{m-1}(q)$};
            \node at (-0.9, -1) {$P_{m-2}(q)$};
            \node at (-0.9,-2) {$P_0(q)$};

            \node[below] at (1.4,-2.5) {$B(q)$};
        \end{scope}
    \end{tikzpicture}
    \caption{$q > q_k$.}
    \label{figure: structure of intersecting sets (+ve delta)}
        \end{subfigure}
        \caption{The relative structure of the convex hulls of the sets $P_0(q) , \ldots P_m(q) , Q_m(q)$ and $B(q)$ when $|q - q_k| < q_k^{-(m+2)k-3}$, as is proved in Lemma \ref{lemma: layout of endpoint of P and Q}.}
        \label{figure: relative structure of P and Q sets}
\end{figure}

Recall that $L(X)$ and $R(X)$ denote the left and right endpoints respectively of a compact set $X \subset \mathbb{R}$.

\begin{lemma}
    \label{lemma: layout of endpoint of P and Q}
    Let $m \in \mathbb{N}$, $k \geq 4$ and let $|q - q_k| < q_k^{-(m+2)k-3}$.
    \begin{enumerate}[(a)] 
        \item 
        \label{item: layout of endpoints q < q_k}
        If $q < q_k$ then
        $$L(Q_m(q)) < L(P_0(q)) < \cdots < L(P_m(q)) < R(Q_m(q)) < R(P_0(q)) < \cdots < R(P_m(q)).$$
         \item 
         \label{item: layout of endpoints q = q_k}
         If $q = q_k$ then
        $$L(Q_m(q)) < L(P_0(q)) = \cdots = L(P_m(q)) < R(Q_m(q)) < R(P_0(q)) = \cdots = R(P_m(q)).$$
        \item
        \label{item: layout of endpoints q > q_k}
        If $q > q_k$ then 
        $$L(Q_m(q)) < L(P_m(q)) < \cdots < L(P_0(q)) < R(Q_m(q)) < R(P_m(q)) < \cdots < R(P_0(q)).$$        
    \end{enumerate}
\end{lemma}

\begin{proof}
    The proof of this lemma is calculation heavy so we take a moment to separate the proof into the stages the argument takes.
    Each step of the argument handles a different part of the long inequality to be proven, but will handle all three of the cases $q < q_k$, $q = q_k$ and $q > q_k$ together.
    Suppose $m \in \mathbb{N}$, $k \geq 4$, $|q - q_k| < q_k^{-(m+2)k-3}$.
    The inequalities we aim to prove in the cases $q < q_k$, $q = q_k$ and $q > q_k$ are respectively
    $$L(Q_m(q)) \underbrace{ <}_\mathrm{A}  L(P_0(q)) \underbrace{ < \cdots <}_{\mathrm{B}} L(P_m(q)) \underbrace{< R(Q_m(q)) <}_{\mathrm{C}} R(P_0(q)) \underbrace{< \cdots <}_\mathrm{D} R(P_m(q)),$$
    $$L(Q_m(q)) \underbrace{ <}_\mathrm{A}  L(P_0(q)) \underbrace{ = \cdots =}_{\mathrm{B}} L(P_m(q)) \underbrace{< R(Q_m(q)) <}_{\mathrm{C}} R(P_0(q)) \underbrace{= \cdots =}_\mathrm{D} R(P_m(q)),$$
    $$L(Q_m(q)) \underbrace{ <}_\mathrm{A}  L(P_m(q)) \underbrace{ < \cdots <}_{\mathrm{B}} L(P_0(q)) \underbrace{< R(Q_m(q)) <}_{\mathrm{C}} R(P_m(q)) \underbrace{< \cdots <}_\mathrm{D} R(P_0(q)).$$

    We start by proving all relations (B) by finding a general expression for $L(P_i(q))$.
    Lemma \ref{lemma: P and Q sets have a common diameter} tells us that $|P_i(q)|$ takes the common value $D_q$ for all $i$  and so the relations (D) are implied by the relations (B).
    It also tells us that $|Q_m(q)|$ shares this common value, so the relations (A) follow from the relations (C).
    Therefore it suffices to prove the relations (B) and (C).

    \textbf{Relations (B)}

    As in the proof of Lemma \ref{lemma: P and Q sets have a common diameter}, $L(P_i(q))$ is the left endpoint of $g_{q,k}^i(\pi_q(S_{k-1}) + 1)$.
    Since $0 \in \pi_q(S_{k-1})$ is the smallest element of this set, using \eqref{equation: affine transformation fixed point} we see that
    \begin{equation}
        \label{equation: left endpoint of P_i(q)}
        L(P_i(q))  = g_{q,k}^i(1)
            = g_{q,k}^i(\pi_q((1^{k-1}0)^\infty) + \epsilon_q)
            = \pi_q((1^{k-1}0)^\infty) + q^{-ki}\epsilon_q.
    \end{equation}
    By \eqref{equation: left endpoint of P_i(q)} we have the following implications.
    If $q < q_k$ then $\epsilon_q < 0$ and
    $$L(P_0(q)) < \cdots < L(P_m(q)).$$
    If $q =q_k$ then $\epsilon_q = 0$ and
    $$L(P_0(q)) = \cdots = L(P_m(q)).$$
    If $q > q_k$ then $\epsilon_q > 0$ and
    $$L(P_m(q)) < \cdots < L(P_0(q)).$$

    \textbf{Relations (C)}

    Given $|q - q_k| < q_k^{-(m+2)k-3}$ we show that
    \begin{enumerate}
        \item if $q < q_k$ then $L(P_m(q)) < R(Q_m(q)) < R(P_0(q))$,
    \end{enumerate}
        and
    \begin{enumerate}
        \item[2.] if $q \geq q_k$ then $L(P_0(q)) < R(Q_m(q)) < R(P_m(q))$.
    \end{enumerate}
    $R(Q_m(q))$ is the right endpoint of the set $g_{q,k}^m(\pi_q(S_{k-1}))$, so
    \begin{align}
    \label{equation: right endpoint of Q_m}
        R(Q_m(q))  &= g_{q,k}^m(\pi_q(1^\infty)),\nonumber \\
                &= g_{q,k}^m(\pi_q((1^{k-1}0)^\infty) + \pi_q((0^{k-1}1)^\infty)),\nonumber \\
                &= \pi_q((1^{k-1}0)^\infty) + q^{-km}\pi_q((0^{k-1}1)^\infty),\nonumber \\
                &= \pi_q((1^{k-1}0)^\infty) + \frac{q^{-km}}{q^k-1}.
    \end{align}
    To find values for $R(P_i(q))$ for $0 \leq i \leq m$, we use $R(P_i(q)) = L(P_i(q)) + D_q$.
    By \eqref{equation: left endpoint of P_i(q)} and Lemma \ref{lemma: P and Q sets have a common diameter} this gives,
    \begin{align}
        \label{equation: right endpoint of P_i}
        R(P_i(q))  =\pi_q((1^{k-1}0)^\infty) + q^{-ki}\epsilon_q + q^{-(m+1)k + 2} \left( \frac{1}{q-1} - \frac{1}{q^k -1} \right).
    \end{align}
    In the first case, when $q < q_k$ and $\epsilon_q < 0$, by \eqref{equation: left endpoint of P_i(q)} and \eqref{equation: right endpoint of Q_m},
    \begin{equation}
    \label{equation: bound for diameter of B(q)}
    R(Q_m(q)) - L(P_m(q)) = \frac{q^{-km}}{q^k-1} - q^{-km}\epsilon_q
                        > q^{-(m+1)k}
                        > 0.  
    \end{equation}
    The second inequality is only valid when $q < q_k$, but we only use this bound under this premise.
    Also, by \eqref{equation: right endpoint of Q_m} and \eqref{equation: right endpoint of P_i},
    \begin{align*}
        R(P_0(q)) - R(Q_m(q)) &= \left( \epsilon_q + q^{-(m+1)k + 2} \left( \frac{1}{q-1} - \frac{1}{q^k -1} \right) \right) - q^{-km}\left( \frac{1}{q^k-1} \right),\\
                        &= \epsilon_q + q^{-km}\left( q^{-k+2}\left( \frac{1}{q-1} - \frac{1}{q^k -1} \right) - \frac{1}{q^k-1} \right).
    \end{align*}
    Since $\epsilon_q > -q_k^{-(m+1)k + 1}$ whenever $|q-q_k| < q_k^{-(m+2)k-3}$, according to Lemma \ref{lemma: bounding epsilon in terms of q}, by factoring out $q^{-km}$, the above expression is positive if
    \begin{equation}
        \label{equation: positive expression in q}
        \left(q^{-k+2}\left(\frac{1}{q-1} - \frac{1}{q^k-1}\right) - \frac{1}{q^k-1} \right) - q^{-k+1} > 0.
    \end{equation}
    This inequality is made clear by interpreting the terms as projections of sequences in base $q$.
    In this way, the left hand side of \eqref{equation: positive expression in q} is equal to
    $$   \pi_q(0^{k-2}(1^{k-1}0)^\infty) - \pi_q((0^{k-1}1)^\infty) - \pi_q(0^{k-2}10^\infty) 
      =  \pi_q(0^{k-2}(001^{k-3}0)(101^{k-3}0)^\infty),$$
    which is obviously positive whenever $k \geq 4$. 
    
    In the second case, when $q \geq q_k$ and $\epsilon_q \geq 0$, using Lemma \ref{lemma: bounding epsilon in terms of q} again we know $\epsilon_q < q_k^{-(m+2)k+1}$, and by \eqref{equation: left endpoint of P_i(q)} and \eqref{equation: right endpoint of Q_m} we have
    \begin{align}
    \label{equation: lower bound of q_k}
        R(Q_m(q)) - L(P_0(q)) &= \pi_q((1^{k-1}0)^\infty) + q^{-km}\left(\frac{1}{q^k-1}\right) - (\pi_q((1^{k-1}0)^\infty) + \epsilon_q), \nonumber \\
                        &= \frac{q^{-km}}{q^k-1} - \epsilon_q, \nonumber \\
                        &> q^{-(m+1)k} - q_k^{-(m+2)k+1}, \nonumber \\
                        &> q_k^{-(m+1)k-1}, \\
                        &>0. \nonumber
    \end{align}
    The lower bound of $q_k^{-(m+1)k-1}$ from step \eqref{equation: lower bound of q_k} follows from a routine calculation.
    Also, by \eqref{equation: right endpoint of P_i} and \eqref{equation: right endpoint of Q_m},
    \begin{align*}
        R(P_m(q)) - R(Q_m(q)) &= \left(q^{-km}\epsilon_q + q^{-(m+1)k + 2} \left( \frac{1}{q-1} - \frac{1}{q^k -1} \right)\right) - q^{-km}\left(\frac{1}{q^k-1}\right),\\
                        &> q^{-km}\epsilon_q + q^{-km}\left( q^{-k+2}\left( \frac{1}{q-1} - \frac{1}{q^k -1} \right) - \frac{1}{q^k-1} \right),
    \end{align*}
    and it follows from our analysis of \eqref{equation: positive expression in q} that this is positive.
    This completes the proof.
\end{proof}

\subsection{Bounding $\beta_q$ below}
\label{subsection: bounding beta below}

Recall that 
$$B(q) = \cap_i \mathrm{conv}(P_i(q)) \cap \mathrm{conv}(Q_m(q)),$$
and
$$\beta_q = \min \left\{ \frac{1}{4} , \frac{|B(q)|}{\max\{|P_0(q)| , \ldots , |P_m(q)| , |Q_m(q)|\}} \right\}.$$

\begin{lemma}
    \label{lemma: bounding beta below}
    If $m \in \mathbb{N}$, $k \geq 2$ and $|q - q_k| < q_k^{-(m+2)k-3}$, then $\beta_q > 1/8$.
\end{lemma}

\begin{proof}
Let $m \in \mathbb{N}$, $k \geq 2$ and $|q - q_k| < q_k^{-(m+2)k-3}$.
Using Lemma \ref{lemma: layout of endpoint of P and Q}, we know that if $q < q_k$, $B(q)$ is given by the interval $[L(P_m(q)) , R(Q_m(q))]$ and if $q \geq q_k$, $B(q)$ is given by the interval $[L(P_0(q)) , R(Q_m(q))]$.
By Lemma \ref{lemma: P and Q sets have a common diameter}, $\max \{ |P_0(q)| , \ldots , |P_m(q)| , |Q_m(q)|\} = D_q$ (their common value), so $\beta_q = \min \left\{ \frac{1}{4} , \frac{|B(q)|}{D_q} \right\}$.

If $q < q_k$ then $q^{-1} > q_k^{-1}$ and $|B(q)| = R(Q_m(q)) - L(P_m(q))$.
Recall that $|P_i(q)| = D_q$ for all $i = 0 , \ldots , m$, and it can be checked using Lemma \ref{lemma: P and Q sets have a common diameter} that $D_q < q_k^{-(m+1)k+3}$.
It follows from \eqref{equation: bound for diameter of B(q)} and the fact that $q^{-(m+1)k} > q_k^{-(m+1)k}$ that $R(Q_m(q)) - L(P_m(q)) > q_k^{-(m+1)k}$ and hence $\beta_q > q_k^{-3} > \frac{1}{8}$.

If $q \geq q_k$ then $q^{-1} \leq q_k^{-1}$ and $|B(q)| = R(Q_m(q)) - L(P_0(q))$.
We know from \eqref{equation: lower bound of q_k} that $R(Q_m(q)) - L(P_0(q)) > q_k^{-(m+1)k-1}$.
Using Lemma \ref{lemma: projections decrease with q} we have $\pi_q((1^{k-1}0)^\infty) \leq \pi_{q_k}((1^{k-1}0)^\infty) = 1$, so Lemma \ref{lemma: P and Q sets have a common diameter} implies that $D_q < q_k^{-(m+1)k+2}$, so $\beta_q > q_k^{-3} > \frac{1}{8}$.
Hence if $|q - q_k| < q_k^{-(m+2)k-3}$ then $\beta_q > \frac{1}{8}$.
\end{proof}

\subsection{Proof of Theorem \ref{theorem: main theorem}}
\label{subsection: Proof of main theorem}

The above results allow us to prove Theorem \ref{theorem: main theorem}.

\begin{proof}[Proof of Theorem \ref{theorem: main theorem}]
    Let $m \in \mathbb{N}$, $k \geq K_m$ and $|q - q_k| < q_k^{-(m+2)k-3}$.
    As discussed in Subsection \ref{subsection: Proof outline main theorem}, to prove Theorem \ref{theorem: main theorem} it suffices to show that the collection of compact subsets of $\mathbb{R}$ given by $\{P_0(q) , \ldots P_m(q) , Q_m(q)\}$ satisfy the assumptions of Theorem \ref{theorem: FandY}.
    We know that $B(q) = \cap_{i=1}^m \mathrm{conv}(P_i(q)) \cap Q_m(q)$ is nonempty by Lemma \ref{lemma: layout of endpoint of P and Q}, and we know that each of $P_0(q) , \ldots , P_m(q) , Q_m(q)$ has thickness at least $q^{k-4}$ by Lemmas \ref{lemma: thickness of S_k} and \ref{lemma: bound thickness of P_i below}.
    By Lemma \ref{lemma: bounding beta below} it remains to show that there exists some $c \in (0,1)$ such that the following inequality holds
    \begin{equation}
        \label{equation: applying FandY with beta=1/8}
      (m+2)(q^{k-4})^{-c} \leq \frac{1}{(432)^2}8^{-c}(1-8^{c-1}).  
    \end{equation}
    Given any $c \in (0,1)$, since $\lim_{k \rightarrow \infty}(q^{k-4})^{-c} = 0$, and
    the right hand side of \eqref{equation: applying FandY with beta=1/8} is independent of $k$, we know there is some value of $k$ for which \eqref{equation: applying FandY with beta=1/8} holds.
    Hence we can arbitrarily\footnote{Some numerical testing showed that $c = \frac{19}{20}$ is a reasonable choice and the authors do not believe that we can strengthen Theorem \ref{theorem: main theorem} by choosing a different value of $c$.} fix $c = \frac{19}{20}$ and find the smallest value of $k$ in terms of $m$ for which \eqref{equation: applying FandY with beta=1/8} holds.
    Notice that $K_m$ is nondecreasing with $m$ and $K_1 = \left\lceil \frac{20}{19}\left( \log_{1.999}(3) + 24 \right) + 4 \right\rceil = 31$.
    Therefore $K_m \geq 31$ for all $m \in \mathbb{N}$ and because $q_{31} > q_{10} > 1.999$ (by direct computation), if $k \geq K_m$, then \eqref{equation: applying FandY with beta=1/8} is implied by
    \begin{equation}
        \label{equation: applying FandY with c}
      (m+2)(1.999)^{(4-k)\frac{19}{20}} \leq \frac{1}{(432)^2}(8^{-\frac{19}{20}})(1-8^{-\frac{1}{20}}).  
    \end{equation}
    The right hand side of \eqref{equation: applying FandY with c} is bounded below by $7.3389 \times 10^{-8}$.
    Using $\log_{1.999}(7.3389 \times 10^{-8}) > -24$, a sufficient condition on $k$ to satisfy \eqref{equation: applying FandY with c} is
    $$k \geq \frac{20}{19}\left( \log_{1.999}(m+2) + 24 \right) + 4,$$
    which is obviously true when $k \geq K_m$.
    This completes the proof.
\end{proof}



\section{Proof of Theorem \ref{theorem: B_3 theorem}}
\label{section: proof of B_3 theorem}

We first outline the proof of Theorem \ref{theorem: B_3 theorem}.

\subsection{Proof outline}
\label{subsection: sketch proof of B_3 theorem}

Using Proposition \ref{proposition: intersection gives nonempty B}, we know that for any $q \in (G,2)$, if
$$(\mathcal{U}_q + 1) \cap g_{q,k}((\mathcal{U}_q + 1) \cap (\mathcal{U}_q)) \neq \emptyset,$$
then $q \in \mathcal{B}_3$.
Let $k \geq 9$.
If $|q - q_k| \leq q_k^{-2k-6}$ then it can be checked that $q > q_{k-1}$ and hence $\pi_q(S_{k-1}) \subset \mathcal{U}_q$ by Lemma \ref{lemma: containment in U_q}\ref{lemma item: U_q contains S_k}.
For each $q$ such that $|q - q_k| < q_k^{-2k-6}$, we construct a family of sets $A_q \subset \{0,1\}^\mathbb{N}$ such that $\pi_q(A_q) \subset (\mathcal{U}_q + 1) \cap \mathcal{U}_q$ whenever $q > q_9$.
This means that if
\begin{equation}
    \label{equation: B_3 nonempty U_q form}
    (\pi_q(S_{k-1}) + 1) \cap g_{q,k}(\pi_q(A_q)) \neq \emptyset,
\end{equation}
then $q \in \mathcal{B}_3$.
Hence to prove Theorem \ref{theorem: B_3 theorem} it suffices to prove that \eqref{equation: B_3 nonempty U_q form} holds when either
\begin{enumerate}[(a)]
    \item
    \label{item: theorem condition a}
    $k \geq 10$ and $|q - q_k| \leq q_k^{-2k-6}$,
\end{enumerate}
or
\begin{enumerate}[(b)]
    \item
    \label{item: theorem condition b}
    $0 < q - q_9 \leq q_9^{-24}$.
\end{enumerate} 
Since \eqref{equation: B_3 nonempty U_q form} concerns an intersection of only two sets, we can employ Theorem \ref{theorem: Newhouse} to conclude their intersection is nonempty.
To do this, we require the sets in question to be interleaved and for the product of their thicknesses to be at least 1.
By Lemma \ref{lemma: thickness of S_k} we know that $\tau(\pi_q(S_{k-1})) > q^{k-4}$ when $k \geq 4$ and $q > q_{k-1}$ and it has been shown in previous work by the authors \cite[Proposition 3.13]{baker2023cardinality} that the set $A_q$ we will construct satisfies
\begin{equation}
    \label{equation: thickness of pi_q(A_q) bound}
    \tau(\pi_q(A_q)) > q^{-5},
\end{equation}
when $q > q_9$.
Recall by the definition of $g_{q,k} = f_1^{-(k-1)} \circ f_0^{-1}$ that $g_{q,k}$ is an affine map
and hence preserves thickness, so
\begin{equation*}
    \label{equation: thickness product at least 1}
    \tau(\pi_q(S_{k-1}) + 1) \times \tau(g_{q,k}(\pi_q(A_q))) > q^{k-4} \times q^{-5} \geq 1,
\end{equation*}
whenever $k \geq 9$, $q > q_9$.
Therefore the thickness condition of Theorem \ref{theorem: Newhouse} is satisfied for both cases \ref{theorem item: B_3 theorem at least 10} and \ref{theorem item: B_3 theorem 9} of Theorem \ref{theorem: B_3 theorem} and it only remains to prove that the sets in \eqref{equation: B_3 nonempty U_q form} are interleaved for $q$ in the required range.
We do this using the notion of $\epsilon$-strong interleaving, introduced in Subsection \ref{subsection: thickness and interleaving}.

Let $k \geq 9$, $|q - q_k| \leq q_k^{-2k-6}$.
We show the following.
\begin{enumerate}
    \item 
    \label{item: B_3 theorem outline strongly interleaved}
    The sets $(\pi_{q_k}(S_{k-1}) + 1)$ and $g_{q_k,k}(\pi_{q_k}(A_{q_k}))$ are $q_k^{-2k-4}$-strongly interleaved.
    \item
    \label{item: B_3 theorem outline bounded Hausdorff distance}
    Both $d_{\mathrm{H}}((\pi_{q_k}(S_{k-1}) + 1) , (\pi_q(S_{k-1}) + 1))$ and $d_{\mathrm{H}}(g_{q_k,k}(\pi_{q_k}(A_{q_k})) , g_{q,k}(\pi_q(A_q)))$ are less than $q_k^{-2k-4}$.
\end{enumerate}
By the definition of $\epsilon$-strong interleaving, these results tell us that $(\pi_q(S_{k-1}) + 1)$ and $g_{q,k}(\pi_q(A_q))$ are interleaved.
To prove the first item, we will find four points in $(\pi_{q_k}(S_{k-1}) + 1) \cap g_{q_k,k}(\pi_{q_k}(A_{q_k}))$ that are sufficiently far apart and then apply Lemma \ref{lemma: strong interleaving}.
Since the thickness condition \eqref{equation: thickness of pi_q(A_q) bound} only holds when $q > q_9$, the complete argument only applies in the smaller domain where either \ref{item: theorem condition a} or \ref{item: theorem condition b} is true.
Hence Theorem \ref{theorem: Newhouse} applies to \eqref{equation: B_3 nonempty U_q form} when either \ref{item: theorem condition a} or \ref{item: theorem condition b} is true and the proof is complete.

\subsection{Construction of $A_q$}
\label{subsection: construction of A_q}

In this subsection we construct the set $A_q$ which will have the important property that for $q > q_9$, $\pi_q(A_q) \subset \mathcal{U}_q \cap (\mathcal{U}_q + 1)$ and $\tau(\pi_q(A_q)) > q^{-5}$.
Define the set 
$$W_2 = \{ (-1 0) , (0 -\!1) , (00) , (01) , (10) \},$$
and write $W_2^\mathbb{N}$ for the set of infinite concatenations of elements of $W_2$.
For any $k \in \mathbb{N}$ and $\seqj{\delta}{1}{k} \in \{-1, 0 ,1\}^k$, define the notation
$$\seqj{\delta}{1}{k} W_2^\mathbb{N} = \{ \seqj{\epsilon}{1}{\infty} \in \{-1, 0,1\}^\mathbb{N} : \seqj{\epsilon}{1}{k} = \seqj{\delta}{1}{k} \ \mathrm{and} \ \seqj{\epsilon}{k+1}{\infty} \in W_2^\mathbb{N} \} .$$

\begin{lemma}
    \label{lemma: expansion of 1 has prefix 1^k0^k+4}
    If $k \geq 9$ and $|q - q_k| \leq q_k^{-2k-6}$ then there exists a sequence $\seqj{c}{1}{\infty} \in 1^k0^{k+4} W_2^\mathbb{N}$ such that $1 = \pi_q(\seqj{c}{1}{\infty})$.
\end{lemma}

In words this lemma is claiming the existence of an expansion of $1$ in base $q$ which agrees with the expansion of $1$ in base $q_k$, given by $1^k 0^\infty$, for the first $2k+4$ digits, and moreover this expansion has a tail in $W_2^\mathbb{N}$.
We will see that this lemma is required to show that the sets $\pi_q(A_q)$ and $\pi_{q_k}(A_{q_k})$ remain sufficiently close in the Hausdorff metric when $|q - q_k| < q_k^{-2k-6}$.
The intuition here is that since the sets $A_q$ will be shown to depend upon a given fixed expansion of $1$ in base $q$, we require the expansions of $1$ themselves to be nearby in the sense that they agree on a large initial segment.

\begin{proof}
    For all $q \in (1,2)$, define $H_q$ to be the interval $[\pi_q((-10)^\infty) , \pi_q((10)^\infty)]$.
    We want to show that if $k \geq 9$ and $|q - q_k| < q_k^{-2k-6}$ then $1$ satisfies 
    \begin{equation}
    \label{equation: condition in H_q}
        f_0^{k+4} \circ f_1^k(1) \in H_q.
    \end{equation}
    It was shown in \cite[Lemma 3.8]{baker2023cardinality} that if $x \in H_q$ then $x$ admits a base $q$ expansion in $W_2^\mathbb{N}$.
    Therefore, if \eqref{equation: condition in H_q} holds there is a sequence $\seqj{\epsilon}{1}{\infty} \in W_2^\mathbb{N}$ such that $f_0^{k+4} \circ f_1^k(1) = \pi_q(\seqj{\epsilon}{1}{\infty})$.
    By Lemma \ref{lemma: relation between maps and projections extension}, this tells us that $1 = \pi_q(1^k0^{k+4}\seqj{\epsilon}{1}{\infty})$.
    Hence we know there is a sequence $\seqj{c}{1}{\infty} \in 1^k 0^{k+4} W_2^\mathbb{N}$ with $1 = \pi_q(\seqj{c}{1}{\infty})$, and so the lemma holds.
    Setting $\seqj{c}{1}{\infty} \in \{0,1\}^\mathbb{N}$ to be any sequence with $1 = \pi_q(\seqj{c}{1}{\infty})$, we can observe using Lemma \ref{lemma: relation between maps and projections extension} again that \eqref{equation: condition in H_q} is equivalent to
    $$\pi_q(\seqj{c}{1}{\infty} \in [\pi_q(1^k0^{k+4} (-1 0)^\infty) , \pi_q(1^k0^{k+4} (1 0)^\infty)],$$
    that is,
    $$\pi_q(1^k0^{k+4} (-1 0)^\infty) \leq \pi_q(\seqj{c}{1}{\infty}) \leq  \pi_q(1^k0^{k+4} (10)^\infty),$$
    and we have the equivalent inequalities
    $$\pi_q(0^{2k+4}(-10)^\infty) \leq \pi_q(\seqj{c}{1}{\infty}) - \pi_q(1^k0^\infty) \leq \pi_q(0^{2k+4}(10)^\infty),$$
    \begin{equation}
    \label{equation: difference of projection in different bases}
    |\pi_{q_k}(1^k0^\infty) - \pi_q(1^k 0^\infty)| \leq \pi_q(0^{2k+4} (10)^\infty),
    \end{equation}
    where we have replaced $\pi_q(\seqj{c}{1}{\infty})$ with $\pi_{q_k}(1^k 0^\infty)$ in the last step since both are equal to $1$.
    It can be checked that \eqref{equation: difference of projection in different bases} is equivalent to
    \begin{equation}
    \label{equation: difference of projections algebraically}
    \left|\left( \frac{1}{q_k-1} - \frac{1}{q-1}\right) - \left( \frac{q_k^{-k}}{q_k-1} - \frac{q^{-k}}{q-1}\right) \right| \leq q^{-2k-4}\frac{q}{q^2-1}.
    \end{equation}
    Both terms in brackets always have the same sign and the latter has smaller magnitude, so we can ignore it and deduce that \eqref{equation: difference of projections algebraically} is implied by
    $$\frac{|q - q_k|}{(q-1)(q_k-1)} \leq q^{-2k-4}\frac{q}{q^2-1}.$$
    Since $\frac{q^2}{q^2-1} >1$ for all $q \in (1,2)$, \eqref{equation: condition in H_q} is true if
    $$|q-q_k| < q^{-2k-5}(q_k-1)(q-1).$$
    Using $|q - q_k| < q_k^{-2k-6}$, this is implied by
    \begin{equation}
    \label{equation: expansion of 1 inequality}
        1 < q_k \left(\frac{q_k}{q}\right)^{2k+5}(q_k-1)(q-1),
    \end{equation}
    and this is what we aim to prove.
    
    Let $k \geq 9$ and $|q - q_k| \leq q_k^{-2k-6}$.
    If $q_k > q$ then $(\frac{q_k}{q}) > 1$ so \eqref{equation: expansion of 1 inequality} is a result of $q_k > 1.9$ and $(q_k-1)(q-1) > 0.9$ (which can be easily checked).
    If $q_k < q$ set $\delta = q - q_k > 0$ and observe that 
    \begin{equation}
    \label{equation: delta inequality}
    \delta/q < \delta \leq q_k^{-2k-6}.
    \end{equation}
    Since $k \geq 9$ we know that $q_k > 1.998$ (by direct computation).
    Bernoulli's inequality tells us that if $2k+5 \geq 1$ and $0 \leq \delta/q \leq 1$ (which are true under our assumptions) then
    $$\left(\frac{q_k}{q}\right)^{2k+5} = (1- \delta/q)^{2k+5} > (1 - (2k+5)q_k^{-2k-6}) > 0.99999,$$
    where the first inequality makes use of \eqref{equation: delta inequality} and the final inequality follows from a direct computation estimating $(2k+5)(1.99)^{-2k-6} > 1.55 \times 10^{-6}$ at $k=9$, and the fact that $(2k+5)q_k^{-2k-6}$ is decreasing as $k$ increases.
    Because $q_k > 1.998$ we also know that
    $$ (q_k - 1)(q-1) > \left(\frac{499}{500}\right)^2 > 0.996,$$
    so the right hand side of \eqref{equation: expansion of 1 inequality} is bounded below by $(1.998)(0.99999)(0.996) > 1$.
    Therefore \eqref{equation: expansion of 1 inequality} holds when $k \geq 9$ and $|q - q_k| < q_k^{-2k-6}$, which completes the proof.
\end{proof}

Let $k \geq 9$, and let $|q - q_k| \leq q_k^{-2k-6}$.
We associate to $q$ a so called \textit{fixed expansion of $1$}, given by $\seqj{c}{1}{\infty} \in 1^k 0^{k+4}W_2^\mathbb{N}$ and satisfying $\pi_q(\seqj{c}{1}{\infty}) = 1$ which we know exists by Lemma \ref{lemma: expansion of 1 has prefix 1^k0^k+4}.
There may be several possible choices for the sequence $\seqj{c}{1}{\infty}$, in this case we just pick one arbitrarily.
As in \cite{baker2023cardinality}, we let $J$ be the set of zeros of the sequence $\seqj{c}{1}{\infty}$:
$$ J = \{j \in \mathbb{N} : c_j = 0\}, $$
which we enumerate as follows:
$J = \{j_0, j_1 , \ldots\}$ where $j_0 < j_1 < \cdots$.
Define 
$$\Jfixedone = \{j_n \in J : n =4m + 1 \ \mathrm{for \ some \ } m \in \mathbb{N}\},$$
$$ \Jfixedzero = \{j_n \in J : n = 4m+3 \ \mathrm{for \ some \ } m \in \mathbb{N}\}, $$
and
$$\Jfree = \{j_n \in J : n \ \mathrm{is \ even} \}.$$
If $j \in \Jfree$ then we call $j$ a \textit{free zero} of $\seqj{c}{1}{\infty}$.
If $j \in \Jfixedone \cup \Jfixedzero$ or if $c_j \in \{-1,1\}$ then $j$ is called a \textit{fixed index}.
Clearly, if $j$ is not a fixed index then it must be a free zero.
Let $1^k 0^{k+4} \seqj{c}{2k+5}{\infty} \in 1^k 0^{k+4} W_2^\mathbb{N}$ be the fixed expansion of $1$ associated to $q$.
The set $A_q$ consists of the sequences $\seqj{a}{1}{\infty} \in \{0,1\}^\mathbb{N}$ which satisfy the following properties:
\begin{enumerate}
    \item $a_j = 1$ if $c_j = 1$,
    \item $a_j = 0$ if $c_j = -1$,
    \item $a_j = 1$ if $j \in \Jfixedone$ and
    \item $a_j = 0$ if $j \in \Jfixedzero$.
\end{enumerate}
Notice that there are no restrictions on the value of $a_j$ if $j \in \Jfree$.
Notice also that $A_q$ depends on $\seqj{c}{1}{\infty}$ and recall that this sequence was itself an arbitrary choice.
The following lemma is an immediate consequence of Proposition 3.10 from \cite{baker2023cardinality}.

\begin{lemma}
    \label{lemma: projection of A_q contained in U_q and U_q + 1}
    If $q \in (q_9, 2)$ then $\pi_q(A_q) \subset \mathcal{U}_q \cap (\mathcal{U}_q + 1)$.
\end{lemma}

Heuristically, this lemma is proved by showing that the strings $(01^9)$ and $(10^9)$ are forbidden in elements of $A_q$, so $A_q \subset S_9$.
Then, if $q \in (q_9,2)$, we know that $\pi_q(S_9) \subset \mathcal{U}_q$ and hence $\pi_q(A_q) \subset \mathcal{U}_q$.
For any $\seqj{a}{1}{\infty} \in A_q$, let $\seqj{b}{1}{\infty} = (a_j - c_j)_{j=1}^\infty$, then $\pi_q(\seqj{b}{1}{\infty}) = \pi_q(\seqj{a}{1}{\infty}) - 1 $.
With this, the containment of $\pi_q(A_q)$ in $(\mathcal{U}_q +1)$ is a result of $\seqj{b}{1}{\infty}$ also being contained in $S_9$.

As mentioned in the outline in Subsection \ref{subsection: sketch proof of B_3 theorem}, the authors previously proved the following lemma \cite[Proposition 3.13]{baker2023cardinality}.

\begin{lemma}
    \label{lemma: thickness of pi_q(A_q)}
    If $q > q_9$ then $\tau(\pi_q(A_q)) > q^{-5}$.
\end{lemma}

\subsection{Hausdorff distances and interleaving}

Using Lemma \ref{lemma: projection of A_q contained in U_q and U_q + 1}, and Lemma \ref{lemma: containment in U_q}\ref{lemma item: U_q contains S_k}, the following lemma is an immediate corollary to Proposition \ref{proposition: intersection gives nonempty B}.

\begin{lemma}
    \label{lemma: corollary to nonempty intersection}
    If $k \geq 10$, $q > q_{k-1}$ and the intersection
    $$(\pi_q(S_{k-1}) + 1) \cap g_{q,k}(\pi_q(A_q)),$$
    is nonempty then $q \in \mathcal{B}_3$.
\end{lemma}

We prove the following result on $\epsilon$-strong interleaving with the intention of applying Lemma \ref{lemma: strong interleaving}.

\begin{lemma}
    \label{lemma: sets are strongly interleaved}
    If $k \geq 9$
    then the sets $(\pi_{q_k}(S_{k-1}) + 1)$ and $g_{q_k,k}(\pi_{q_k}(A_{q_k}))$ are $q_k^{-2k-4}$-strongly interleaved.
\end{lemma}

\begin{proof}
    Let $k \geq 9$.
    By the definition of $q_k$, $1 = \pi_{q_k}(1^k 0^\infty)$ and we label $(c_{k,j})_{j=1}^\infty = 1^k 0^\infty$.
    By the definition of $A_{q_k}$ when associated with the fixed expansion of $1$ given by $(c_{k,j})_{j=1}^\infty$, we know that for any $m \in \mathbb{N}_{\geq 0}$, the indices $k+(2m+1)$ are free $0$s, the indices $k + (4m+2)$ are fixed $1$s, and the indices $k + (4m+4)$ are fixed $0$s.
    Using this we construct four points in $\pi_{q_k}(A_{q_k})$, with the intention of applying Lemma \ref{lemma: strong interleaving}.

    Consider the points $a^{(i)} = \pi_q(\seqj{a^i}{1}{\infty})$ where $i \in \{1, 2, 3,4\}$ and $a^{(1)} < a^{(2)} < a^{(3)} < a^{(4)}$ are given by
    \begin{align*}
        a^{(1)} &= \pi_{q_k}(1^k (0100)^\infty) = q_k^{-k}\frac{1}{q_k^4-1}(q_k^2) , \\
        a^{(2)} &= \pi_{q_k}(1^k (0110)^\infty) = q_k^{-k}\frac{1}{q_k^4-1}(q_k+q_k^2), \\
        a^{(3)} &= \pi_{q_k}(1^k (1100)^\infty) = q_k^{-k}\frac{1}{q_k^4-1}(q_k^2+q_k^3), \\
        a^{(4)} &= \pi_{q_k}(1^k (1110)^\infty) = q_k^{-k}\frac{1}{q_k^4-1}(q_k + q_k^2 +q_k^3).
    \end{align*}    
    We check that $g_{q_k,k}(a^{(i)}) \in (\pi_{q_k}(S_{k-1}) + 1) \cap g_{q_k,k}(\pi_{q_k}(A_{q_k}))$ for all $i \in \{1,2,3,4\}$.
    Since $a^{(i)} \in \pi_{q_k}(A_{q_k})$ by construction, it is obvious that $g_{q_k,k}(a^{(i)}) \in g_{q_k,k}(\pi_{q_k}(A_{q_k}))$ for all $i \in \{1,2,3,4\}$.
    For the other containment, we show that $g_{q_k,k}(a^{(i)}) - 1 \in \pi_{q_k}(S_{k-1}) $.
    Recall that $1 = \pi_{q_k}(1^k 0^\infty)$, and since $q = q_k$, we can replace\footnote{This is a consequence of the definition of $q_k$, which is equivalent to $q_k^{-1} = q_k^{-2} + \cdots + q_k^{-(k+1)}$.} $01^k$ with $10^k$ at any point in the expansion.
    Recall also that \eqref{equation: g_qk prefixes by a sequence} tells us that $g_{q,k}$ acts to prefix base $q$ expansions by $(1^{k-1}0)$.
    This allows us to write
    \begin{align*}
        g_{q_k,k}(a^{(i)}) - 1 &= \pi_{q_k}(1^{k-1} 0 1^k \seqj{a^{(i)}}{k+1}{\infty}) - \pi_{q_k}(1^k 0^\infty), \\
                     &= \pi_{q_k}(1^k 0^k \seqj{a^{(i)}}{k+1}{\infty}) - \pi_{q_k}(1^k 0^\infty), \\
                     &= \pi_{q_k}(0^{2k} \seqj{a^{(i)}}{k+1}{\infty}).
    \end{align*}
    For all $i \in \{1,2,3,4\}$, the tail $\seqj{a^{(i)}}{k+1}{\infty}$ avoids $0^4$ and $1^4$ because they are sequences in $A_q$.
    This allows us to declare that $0^{2k}\seqj{a^{(i)}}{k+1}{\infty}$ is contained in $S_{k-1}$.
    Therefore, $g_{q_k,k}(a^{(i)}) \subset (\pi_{q_k}(S_{k-1}) + 1)$ for all $i \in \{1,2,3,4\}$.
    The four points $g_{q_k,k}(a^{(i)})$ where $i \in \{1,2,3,4\}$ are given by
    \begin{align*}
        g_{q_k,k}(a^{(1)}) &= \pi_{q_k}(1^{k-1} 0 1^k (0100)^\infty), \\
        g_{q_k,k}(a^{(2)}) &= \pi_{q_k}(1^{k-1} 0 1^k (0110)^\infty), \\
        g_{q_k,k}(a^{(3)}) &= \pi_{q_k}(1^{k-1} 0 1^k (1100)^\infty), \\
        g_{q_k,k}(a^{(4)}) &= \pi_{q_k}(1^{k-1} 0 1^k (1110)^\infty),
    \end{align*}
    and we seek the minimum distance between them.
    Since $a^{(1)} < a^{(2)} < a^{(3)} < a^{(4)}$ and $g_{q,k}$ is increasing, we need only consider the distances $(g_{q,k}(a^{(i+1)}) - g_{q,k}(a^{(i)}))$ for $i \in \{2,3,4\}$.
Observe that since $q_k^3 - q_k = q_k(q_k^2-1) > q_k$, the smallest of 
     \begin{align*}
        (a^{(2)} - a^{(1)}) &= q_k^{-k}\frac{1}{q_k^4-1}q_k, \\
        (a^{(3)} - a^{(2)}) &= q_k^{-k}\frac{1}{q_k^4-1}(q_k^3 - q_k), \\
        (a^{(4)} - a^{(3)}) &= q_k^{-k}\frac{1}{q_k^4-1}q_k,
    \end{align*}
is given by $\frac{q_k^{-k+1}}{q_k^4 - 1}$.
The smallest over $i \in \{1,2,3,4\}$ of $g_{q_k,k}(a^{(i+1)}) - g_{q_k,k}(a^{(i)}) = q_k^{-k}(a^{(i+1)} - a^{(i)})$ is then given by $q_k^{-k}(\frac{q_k^{-k+1}}{q_k^4-1})$.
Therefore, by Lemma \ref{lemma: strong interleaving}, $(\pi_{q_k}(S_{k-1}) + 1)$ and $g_{q_k,k}(\pi_{q_k}(A_{q_k}))$ are $\epsilon$-strongly interleaved where $2 \epsilon = \frac{q_k^{-2k+1}}{q_k^4-1}$.
So 
$$\epsilon = \frac{q_k^{-2k+1}}{2(q_k^4-1)} = q_k^{-2k-3} \frac{1}{2(1-q_k^{-4})}.$$

Since $q_k \geq q_9 \geq 1.998 > 2(1-2^{-4}) = 1.875 > 2(1-q_k^{-4})$, we know that $\frac{1}{2(1-q_k^{-4})} > \frac{1}{q_k}$, so $\epsilon > q_k^{-2k-4}$.
Hence we have shown that $(\pi_{q_k}(S_{k-1}) + 1)$ and $g_{q_k,k}(\pi_{q_k}(A_{q_k}))$ are $q_k^{-2k-4}$-strongly interleaved.
\end{proof}
We use the following lemma as a tool in the proofs of Lemmas \ref{lemma: bounding Hausdorff distance pi_q(S_k-1)} and \ref{lemma: hausdorff distance between g_q pi_q A_q}.

\begin{lemma}
    \label{lemma: different projections of same sequence bound}
    Let $q_1, q_2 \in (1,2)$ and $\seqj{\epsilon}{1}{\infty} \in \{-1,0,1\}^\mathbb{N}$, then
    $$|\pi_{q_1}(\seqj{\epsilon}{1}{\infty}) - \pi_{q_2}(\seqj{\epsilon}{1}{\infty})|
    \leq \frac{|q_1 - q_2|}{(q_1-1)(q_2-1)}.$$
\end{lemma}

\begin{proof}
    By the definition of the projection maps,
    $$|\pi_{q_1}(\seqj{\epsilon}{1}{\infty}) - \pi_{q_2}(\seqj{\epsilon}{1}{\infty})| = \left|\sum_{j=1}^\infty\epsilon_j(q_1^{-j} - q_2^{-j})\right| \leq \left|\sum_{j=1}^\infty(q_1^{-j} - q_2^{-j})\right| =\left| \frac{1}{q_1-1} - \frac{1}{q_2-1}\right| = \frac{|q_1 - q_2|}{(q_1-1)(q_2-1)}.$$
\end{proof}

The following two lemmas provide us with upper bounds on the Hausdorff distances between the sets in question when $q$ is in the required range.

\begin{lemma}
    \label{lemma: bounding Hausdorff distance pi_q(S_k-1)}
    If $k \geq 9$ and $|q - q_k| < q_k^{-2k-6}$ then $d_{\mathrm{H}}(\pi_q(S_{k-1}) , \pi_{q_k}(S_{k-1})) < q_k^{-2k-4}$.
\end{lemma}

\begin{proof}
    Let $k \geq 9$, $|q - q_k| < q_k^{-2k-6}$ and let $\seqj{s}{1}{\infty} \in S_{k-1}$ be arbitrary so $\pi_{q_k}(\seqj{s}{1}{\infty}) \in \pi_{q_k}(S_{k-1})$ is arbitrary.
    Then by Lemma \ref{lemma: different projections of same sequence bound}
    \begin{align*}
    |\pi_{q_k}(\seqj{s}{1}{\infty}) - \pi_q(\seqj{s}{1}{\infty})| 
    &\leq \frac{|q-q_k|}{(q_k-1)(q-1)}, \\
    & < \frac{5}{4}|q-q_k|, \\
    & < \frac{5}{4}q_k^{-2k-6}, \\
    & < q_k^{-2k-4}.
    \end{align*}
    Since $q_{9} \approx 1.998 $ the factor of $5/4$ appears as an easy lower bound for $\frac{1}{(q_k-1)(q-1)}$.
    So we have found a point $\pi_q(\seqj{s}{1}{\infty}) \in \pi_q(S_{k-1})$ such that $|\pi_{q_k}(\seqj{s}{1}{\infty}) - \pi_q(\seqj{s}{1}{\infty})| < q_k^{-2k-4}$.
    This argument also proves the converse, that is, for an arbitrary point $\pi_q(\seqj{s}{1}{\infty}) \in \pi_q(S_{k-1})$ the point $\pi_{q_k}(\seqj{s}{1}{\infty}) \in \pi_{q_k}(S_{k-1})$ satisfies $|\pi_{q_k}(\seqj{s}{1}{\infty}) - \pi_q(\seqj{s}{1}{\infty})| < q_k^{-2k-4}$.
    Therefore the desired upper bound on the Hausdorff distance holds.
\end{proof}

\begin{lemma}
    \label{lemma: hausdorff distance between g_q pi_q A_q}
    If $k \geq 9$ and $|q - q_k| < q_k^{-2k-6}$ then $d_{\mathrm{H}}(g_{q,k}(\pi_q(A_q)) , g_{q_k,k}(\pi_{q_k}(A_{q_k}))) < q_k^{-2k-4}$.
\end{lemma}

\begin{proof}
    Let $k \geq 9$, $|q - q_k| < q_k^{-2k-6}$ and let $\seqj{c}{1}{\infty}$ be the fixed expansion of $1$ associated to $q$.
    Recall that $A_q$ is defined with respect to $\seqj{c}{1}{\infty}$.
    Using Lemma \ref{lemma: expansion of 1 has prefix 1^k0^k+4}, $\seqj{c}{1}{\infty}$ and $1^k 0^\infty$ agree on the first $2k+4$ indices, we know that the free and fixed indices on this initial segment coincide.
    Therefore, for every $\seqj{a}{1}{\infty} \in A_{q_k}$, we can find some $\seqj{a'}{1}{\infty} \in A_q$ such that $\seqj{a}{1}{2k+4} = \seqj{a'}{1}{2k+4}$, and vice versa.
    We show that these sequences satisfy the property that 
    \begin{equation}
        \label{equation: bounded Hausdorff distance g_qk set}
        | g_{q_k,k}(\pi_{q_k}(\seqj{a}{1}{\infty})) - g_{q,k}(\pi_q(\seqj{a'}{1}{\infty}))| < q_k^{-2k-4}.
    \end{equation}
    Notice that $\seqj{a}{1}{\infty} \in A_{q_k}$ was arbitrary, so $g_{q_k,k}(\pi_{q_k}(\seqj{a}{1}{\infty}))$ is an arbitrary element of $g_{q_k,k}(\pi_{q_k}(A_{q_k}))$.
    Therefore it suffices to prove that \eqref{equation: bounded Hausdorff distance g_qk set} holds since the converse direction, where we first choose an arbitrary point in $A_q$, proceeds in exactly the same way.

    Let $\seqj{a}{1}{\infty} \in A_{q_k}$ be arbitrary and pick $\seqj{a'}{1}{\infty} \in A_q$ such that $\seqj{a}{1}{2k+4} = \seqj{a'}{1}{2k+4}$.
    Then $g_{q_k,k}(\pi_{q_k}(\seqj{a}{1}{\infty})))$ is given by $\pi_{q_k}(1^{k-1} 0 \seqj{a}{1}{\infty})$ and $g_{q,k}(\pi_{q}(\seqj{a'}{1}{\infty})))$ is given by $\pi_{q}(1^{k-1} 0 \seqj{a'}{1}{\infty})$.
    Observe that by the triangle inequality,
    \begin{multline}
    | \pi_{q_k}(1^{k-1} 0 \seqj{a}{1}{\infty}) - \pi_q(1^{k-1} 0 \seqj{a'}{1}{\infty})| \leq
    |\pi_{q_k}(1^{k-1}0 \seqj{a}{1}{\infty}) - \pi_{q_k}(1^{k-1}0 \seqj{a'}{1}{\infty})| \\+ 
    |\pi_{q_k}(1^{k-1}0 \seqj{a'}{1}{\infty}) - \pi_q(1^{k-1}0 \seqj{a'}{1}{\infty})|.    
    \end{multline}
    Since $\seqj{a}{1}{\infty}$ and $\seqj{a'}{1}{\infty}$ agree on their first $(2k+4)$ entries, we know the sequences $(1^{k-1}0 \seqj{a}{1}{\infty})$ and $(1^{k-1}0 \seqj{a'}{1}{\infty})$ agree on their first $(3k+4)$ entries.
    Therefore, we can bound the first term as follows
    \begin{align*}
    |\pi_{q_k}(1^{k-1}0 \seqj{a}{1}{\infty}) - \pi_{q_k}(1^{k-1}0 \seqj{a'}{1}{\infty})| 
    &\leq q_k^{-3k-4}|\pi_{q_k}(\seqj{a}{2k+5}{\infty}) - \pi_{q_k}(\seqj{a'}{2k+5}{\infty})|,\\
    &\leq q_k^{-3k-4}\left(\frac{1}{q_k-1}\right),  
    \end{align*}
    where the $\frac{1}{q_k-1}$ term appears as an upper bound for the difference between the images of any two sequences under $\pi_{q_k}$, i.e. the diameter of $I_{q_k}$.
    By Lemma \ref{lemma: different projections of same sequence bound}, the second term satisfies
    $$|\pi_{q_k}(1^{k-1}0 \seqj{a'}{1}{\infty}) - \pi_q(1^{k-1}0 \seqj{a'}{1}{\infty})| \leq \frac{|q - q_k|}{(q-1)(q_k-1)}.$$
    Since both $\frac{1}{q_k-1}$ and $(q_k^{-k+2} + \frac{1}{q-1})$ are easily bounded above by $q_k$, we can see that 
    the left hand side of \eqref{equation: bounded Hausdorff distance g_qk set} can be bounded as follows:
    \begin{align*}
    | g_{q_k,k}(\pi_{q_k}(\seqj{a}{1}{\infty})) - g_{q,k}(\pi_q(\seqj{a'}{1}{\infty}))| 
    &\leq q_k^{-3k-4}\left(\frac{1}{q_k-1}\right) + \frac{q_k^{-2k-6}}{(q-1)(q_k-1)}, \\
    &\leq q_k^{-2k-6}\left(\frac{1}{q_k-1}\right)\left(q_k^{-k+2} + \frac{1}{q-1}\right),\\
    &\leq q_k^{-2k-4}.
    \end{align*}
    So we have proved \eqref{equation: bounded Hausdorff distance g_qk set} holds, which proves the lemma.
\end{proof}

\subsection{Proof of Theorem \ref{theorem: B_3 theorem}}

\begin{proof}[Proof of Theorem \ref{theorem: B_3 theorem}]
    Let $k \geq 9$ and $|q - q_k| < q_k^{-2k-6}$.
    By Lemma \ref{lemma: corollary to nonempty intersection}, it suffices to prove that \eqref{equation: B_3 nonempty U_q form} holds in cases \ref{theorem item: B_3 theorem at least 10} and \ref{theorem item: B_3 theorem 9}.
    Using the definition of $\epsilon$-strong interleaving alongside Lemmas \ref{lemma: sets are strongly interleaved}, \ref{lemma: bounding Hausdorff distance pi_q(S_k-1)} and \ref{lemma: hausdorff distance between g_q pi_q A_q}, we know that $(\pi_q(S_{k-1}) + 1)$ and $g_{q,k}(\pi_q(A_q))$ are interleaved.
    If we make the further restriction that $ 0 < q - q_9 < q_9^{-24}$ when $k = 9$, then by Lemmas \ref{lemma: thickness of S_k} and \ref{lemma: thickness of pi_q(A_q)}, we know that $\tau(\pi_q(S_{k-1}) + 1) \times \tau(g_{q,k}(\pi_q(A_q))) > q^{5} \times q^{-5} =1$.
    Hence we can apply Theorem \ref{theorem: Newhouse} to \eqref{equation: B_3 nonempty U_q form} in both cases \ref{theorem item: B_3 theorem at least 10} and \ref{theorem item: B_3 theorem 9} and the proof is complete.
\end{proof}

\section*{Acknowledgments}

The first author was financially supported by an EPSRC New Investigators Award (EP/W003880/1).

\bibliography{Bibliography.bib}

\begin{thebibliography}{10}

\bibitem{BarBakKon2019Entropy}
Rafael Alcaraz~Barrera, Simon Baker, and Derong Kong.
\newblock Entropy, topological transitivity, and dimensional properties of
  unique {$q$}-expansions.
\newblock {\em Trans. Amer. Math. Soc.}, 371(5):3209--3258, 2019.

\bibitem{Baker2014golden}
Simon Baker.
\newblock Generalized golden ratios over integer alphabets.
\newblock {\em Integers}, 14:Paper No. A15, 28, 2014.

\bibitem{Bak2015OnSmall}
Simon Baker.
\newblock On small bases which admit countably many expansions.
\newblock {\em J. Number Theory}, 147:515--532, 2015.

\bibitem{baker2023cardinality}
Simon Baker and George Bender.
\newblock On the cardinality and dimension of the slices of {O}kamoto's
  functions.
\newblock {\em arXiv e-prints}, pages arXiv--2310, 2023.

\bibitem{BakSid2014}
Simon Baker and Nikita Sidorov.
\newblock Expansions in non-integer bases: lower order revisited.
\newblock {\em Integers}, 14:Paper No. A57, 15, 2014.

\bibitem{BakerYuru2023Metric}
Simon Baker and Yuru Zou.
\newblock Metric results for numbers with multiple {$q$}-expansions.
\newblock {\em J. Fractal Geom.}, 10(3-4):243--266, 2023.

\bibitem{DarKat1995Structure}
Zolt\'an Dar\'oczy and Imre K\'atai.
\newblock On the structure of univoque numbers.
\newblock {\em Publ. Math. Debrecen}, 46(3-4):385--408, 1995.

\bibitem{ErJoHo1991}
P\'al Erd\"os, Mikl\'os Horv\'ath, and Istv\'an Jo\'o.
\newblock On the uniqueness of the expansions {$1=\sum q^{-n_i}$}.
\newblock {\em Acta Math. Hungar.}, 58(3-4):333--342, 1991.

\bibitem{ErdJoo1993}
P{\'a}l Erd{\"o}s and Istv{\'a}n Jo\'o.
\newblock On the number of expansions {$1=\sum q^{-n_i}$}. {II}.
\newblock {\em Ann. Univ. Sci. Budapest. E\"otv\"os Sect. Math.}, 36:229--233,
  1993.

\bibitem{erdos1990characterization}
P{\'a}l Erd{\"o}s, Vilmos Komornik, and Istv\'an Jo\'o.
\newblock Characterization of the unique expansions $1 = \sum_{i=1}^\infty
  q^{-n_i}$ and related problems.
\newblock {\em Bulletin de la Soci{\'e}t{\'e} Math{\'e}matique de France},
  118(3):377--390, 1990.

\bibitem{FalYav2022}
Kenneth Falconer and Alexia Yavicoli.
\newblock Intersections of thick compact sets in $\mathbb{R}^d$.
\newblock {\em Math. Z.}, 301(3):2291--2315, 2022.

\bibitem{GlenSid2001}
Paul Glendinning and Nikita Sidorov.
\newblock Unique representations of real numbers in non-integer bases.
\newblock {\em Math. Res. Lett.}, 8(4):535--543, 2001.

\bibitem{KomKonLi2017Hausdorff}
Vilmos Komornik, Derong Kong, and Wenxia Li.
\newblock Hausdorff dimension of univoque sets and devil's staircase.
\newblock {\em Adv. Math.}, 305:165--196, 2017.

\bibitem{KomLor98}
Vilmos Komornik and Paola Loreti.
\newblock Unique developments in non-integer bases.
\newblock {\em Amer. Math. Monthly}, 105(7):636--639, 1998.

\bibitem{Newhouse1970}
Sheldon~E. Newhouse.
\newblock Nondensity of axiom {${\rm A}({\rm a})$} on {$S\sp{2}$}.
\newblock In {\em Global {A}nalysis ({P}roc. {S}ympos. {P}ure {M}ath., {V}ols.
  {XIV}, {XV}, {XVI}, {B}erkeley, {C}alif., 1968)}, volume XIV-XVI of {\em
  Proc. Sympos. Pure Math.}, pages 191--202. Amer. Math. Soc., Providence, RI,
  1970.

\bibitem{parry1960beta}
William Parry.
\newblock On the $\beta$-expansions of real numbers.
\newblock {\em Acta Math. Acad. Sci. Hungar}, 11:401--416, 1960.

\bibitem{renyi1957representations}
Alfr{\'e}d R{\'e}nyi.
\newblock Representations for real numbers and their ergodic properties.
\newblock {\em Acta Math. Acad. Sci. Hungar}, 8(3-4):477--493, 1957.

\bibitem{Sidorov2003}
Nikita Sidorov.
\newblock Almost every number has a continuum of {$\beta$}-expansions.
\newblock {\em Amer. Math. Monthly}, 110(9):838--842, 2003.

\bibitem{sidorov2009expansions}
Nikita Sidorov.
\newblock Expansions in non-integer bases: lower, middle and top orders.
\newblock {\em J. Number Theory}, 129(4):741--754, 2009.

\bibitem{SidVer1998Goldenshift}
Nikita Sidorov and Anatoly Vershik.
\newblock Ergodic properties of the {E}rd\"os measure, the entropy of the
  golden shift, and related problems.
\newblock {\em Monatsh. Math.}, 126(3):215--261, 1998.

\end{thebibliography}
\bibliographystyle{plain}

\end{document}